\documentclass[11pt,fleqn,leqno]{siamltex}

\usepackage{amsmath}
\usepackage{amssymb}
\usepackage{graphicx}
\usepackage{url}
\usepackage{hyperref}
\newtheorem{experiment}[theorem]{Experiment}
\newtheorem{remark}[theorem]{Remark}

\newtheorem{summary}[theorem]{Summary}
\newtheorem{mtheorem}[theorem]{Meta-Theorem}

\newcommand{\R}{\mathbb R}
\newcommand{\C}{\mathbb C}
\newcommand{\eps}{\varepsilon}

\newcommand{\ph}{\phantom}
\newcommand{\vecu}{u}
\newcommand{\vecx}{x}

\newcommand{\mtxa}[2]{
\left[ \begin{array}{#1} #2 \end{array} \right]}

\newcommand{\smtxa}[2]{
{\mbox{\scriptsize
$\left[\!\! \begin{array}{#1} #2 \end{array} \!\! \right]$}}}

\usepackage{xcolor}

\usepackage{lineno}
\newtheorem{conjecture}[theorem]{Conjecture}

\catcode`@=11 % Make @ act like a letter
\font\tenex=cmex10 % math extension
\newdimen\p@renwd
\setbox0=\hbox{\tenex B} \p@renwd=\wd0 % width of the big left [
\def\bmat#1{\begingroup \m@th
 \setbox\z@\vbox{\def\cr{\crcr\noalign{\kern2\p@\global\let\cr\endline}}%
\ialign{$##$\hfil\kern2\p@\kern\p@renwd&\thinspace\hfil$##$\hfil
 &&\quad\hfil$##$\hfil\crcr
 \omit\strut\hfil\crcr\noalign{\kern-\baselineskip}%
 #1\crcr\omit\strut\cr}}%
 \setbox\tw@\vbox{\unvcopy\z@\global\setbox\@ne\lastbox}%
 \setbox\tw@\hbox{\unhbox\@ne\unskip\global\setbox\@ne\lastbox}%
 \setbox\tw@\hbox{$\kern\wd\@ne\kern-\p@renwd\left[\kern-\wd\@ne
 \global\setbox\@ne\vbox{\box\@ne\kern2\p@}%
 \vcenter{\kern-\ht\@ne\unvbox\z@\kern-\baselineskip}\,\right]$}%
 \null\;\vbox{\kern\ht\@ne\box\tw@}\endgroup}
\catcode`@=12 % @ is no longer a letter
\newif\ifMDlatex
\catcode`@=11     % Make @ act like a letter
\def\MD@us#1{\csname#1style\endcsname}
\def\MD@uf#1{\csname#1font\endcsname}
\def\MD@t{text}
\def\MD@s{script}
\def\MD@ss{scriptscript}
\newdimen\MD@unit
\MD@unit\p@
\def\MD@changestyle#1{
  \relax\MD@unit0.1\fontdimen6\MD@uf{#1}0
  \everymath\expandafter{\the\everymath\MD@us{#1}}
}
\def\MD@dot{$\m@th\ldotp$}
\def\MD@palette#1{\mathchoice{#1\MD@t}{#1\MD@t}{#1\MD@s}{#1\MD@ss}}
\def\MD@ddots#1{{\MD@changestyle{#1}%
  \mkern1mu\raise7\MD@unit\vbox{\kern7\MD@unit\hbox{\MD@dot}}%
  \mkern2mu\raise4\MD@unit\hbox{\MD@dot}%
  \mkern2mu\raise \MD@unit\hbox{\MD@dot}\mkern1mu}}%
\def\MD@iddots#1{{\MD@changestyle{#1}%
  \mkern1mu\raise \MD@unit\hbox{\MD@dot}%
  \mkern2mu\raise4\MD@unit\hbox{\MD@dot}%
  \mkern2mu\raise7\MD@unit\vbox{\kern7\MD@unit\hbox{\MD@dot}}}}%
\def\MD@vdots#1{\vbox{\MD@changestyle{#1}%
    \baselineskip4\MD@unit\lineskiplimit\z@
    \kern6\MD@unit\hbox{\MD@dot}\hbox{\MD@dot}\hbox{\MD@dot}}}%
\ifMDlatex
  \DeclareRobustCommand\ddots{\mathinner{\MD@palette\MD@ddots}}%
  \DeclareRobustCommand\iddots{\mathinner{\MD@palette\MD@iddots}}%
  \DeclareRobustCommand\vdots{\mathinner{\MD@palette\MD@vdots}}%
\else
  \def\ddots{\mathinner{\MD@palette\MD@ddots}}%
  \def\iddots{\mathinner{\MD@palette\MD@iddots}}%
  \def\vdots{\mathinner{\MD@palette\MD@vdots}}%
\fi
\catcode`@=12    % @ is no longer a letter
\newcommand{\adots}{\iddots}

\author{Michiel~E.~Hochstenbach\thanks{%
Version \today.
Department of Mathematics and Computer Science,
TU Eindhoven,
PO Box 513, 5600 MB, The Netherlands,
{\tt www.win.tue.nl/$\sim$hochsten}.}
\and
Christian Mehl\thanks{%
Institut f\"{u}r Mathematik, Technische Universit\"at Berlin, Sekretariat MA 4-5,
Stra\ss e des 17.~Juni 136, 10623 Berlin, Germany,
{\tt mehl@math.tu-berlin.de}.
This author has been supported by a Dutch 4TU AMI visitor's grant.}
\and
Bor~Plestenjak\thanks{%
IMFM and Faculty of Mathematics and Physics, University of Ljubljana,
Jadranska 19, SI-1000 Ljubljana, Slovenia,
{\tt bor.plestenjak@fmf.uni-lj.si}.
This author
has been supported by the Slovenian Research Agency (grant N1-0154).}}

\title{Solving singular generalized eigenvalue problems \\
Part III: structure preservation}

\begin{document}
\maketitle

\begin{abstract}
In Parts I and II of this series of papers, three new methods for the computation of eigenvalues of singular pencils were developed:
rank-completing perturbations, rank-projections, and augmentation.
It was observed that a straightforward structure-preserving adaption for symmetric pencils was not possible and it was left as an open question how to address this challenge.
In this Part III, it is shown how the observed issue can be circumvented by using
Hermitian perturbations. This leads to structure-preserving analogues of the three
techniques from Parts I and II for Hermitian pencils (including real symmetric
pencils) as well as for related structures. It is an important feature of these
methods that the sign characteristic of the given pencil is preserved.
As an application, it is shown that the resulting methods can be used to solve systems of bivariate polynomials.
\end{abstract}

\begin{keywords}
Singular symmetric pencil, singular Hermitian pencil, sign characteristic, generalized eigenvalue problem, structure-preserving perturbation, projection, augmentation, symmetric determinantal representations, bivariate polynomial systems.
\end{keywords}

\begin{AMS}
65F15, 15A18, 15A22, % pencils
15A21, 47A55, % canonical forms, perturbation theory for linear operators
65F22 % ill-posedness, regularization
% dimension reduction: not applicable, and no code
\end{AMS}

\pagestyle{myheadings}
\thispagestyle{plain}
\markboth{HOCHSTENBACH, MEHL, AND PLESTENJAK}{COMPUTING EIGENVALUES OF STRUCTURED SINGULAR PENCILS}

\section{Introduction}
\label{sec1}
We consider the real symmetric or complex Hermitian singular $n \times n$ generalized eigenvalue problem
\begin{equation} \label{gep}
Ax = \lambda \, Bx.
\end{equation}
This means that $A^* = A$, $B^* = B$ are real or complex $n\times n$ matrices and
we have $\det(A-\lambda B) \equiv 0$.
In particular, we note that both $A$ and $B$ may be indefinite, that is, may have positive and negative eigenvalues (in addition to having zero eigenvalues since singularity of
the pencil necessarily implies that both $A$ and $B$ are singular as well).

Hermitian pencils occur in many applications, such as in the linear quadratic
optimal control problem, i.e., the problem of solving the system
\[
E\dot{x}=Ax+Bu,\quad x(t_0)=x_0
\]
while simultaneously minimizing the cost function
\[
\int_{t_0}^\infty \big(x(t)^*\,Q\,x(t)+u(t)^*\,R\,u(t) + u(t)^*\,S^*\,x(t) + x(t)^*\,S\,u(t)\big)\,{\rm d}t.
\]
Here, we have $A,E,Q\in\mathbb C^{n,n}$ with $Q^*=Q$; $B,S\in\mathbb C^{n,m}$ and $R^*=R\in\mathbb C^{m,m}$. Furthermore, $x:[t_0,\infty)\,\to\mathbb C^n$ is the state of the system and
$u:[t_0,\infty)\,\to\mathbb C^m$ is the input, $x_0\in\mathbb C^n$ and $t_0\in\mathbb R$; see \cite{LanR95,Meh91}. It is well known that the solution can
be obtained via the solution of the generalized eigenvalue problem
$\mathcal A-\lambda\mathcal B$, where
\[
\mathcal A=\left[\begin{array}{ccc}Q&A^*&S\\ A&0&B\\ S^*&B^*&R\end{array}\right]
\quad\mbox{and}\quad
\mathcal B=\left[\begin{array}{ccc}0&-E^*&0\\ E&0&0\\ 0&0&0\end{array}\right];
\]
see \cite{Meh91} and the references therein.
Then the pencil $\mathcal A-\lambda \, (i\mathcal B)$
is a Hermitian pencil which may be singular, in particular, in the case
of a descriptor system, i.e., when $E$ is singular.

Another application are quadratic matrix polynomials of the form
$\lambda^2 M+\lambda D+K$ arising in vibration analysis \cite{GohLR82,TisM01}, with $M,D,K$ real symmetric or complex Hermitian. As proposed in \cite{HigMMT06b},
these polynomials can be linearized in a structure-preserving manner
by either one of the Hermitian pencils
\[
\left[\begin{array}{cc}D&K\\ K&0\end{array}\right]-
\lambda\left[\begin{array}{cc}M&0\\ 0&-K\end{array}\right]
\quad\mbox{or}\quad\left[\begin{array}{cc}-M&0\\ 0&K\end{array}\right]-
\lambda\left[\begin{array}{cc}0&M\\ M&D\end{array}\right].
\]
If both $M$ and $K$ are singular, then both linearizations are singular Hermitian pencils.
Nevertheless, as shown in \cite{DopN23}, the nonzero finite eigenvalues of
the quadratic matrix polynomial can be computed by solving the singular Hermitian
eigenvalue problem associated with either one of the two linearizations.

Furthermore, many applications in data analysis and machine learning give rise to eigenvalue problem of the form
\begin{equation} \label{dimred}
XFX^T \, z = \lambda \, XGX^T \, z,
\end{equation}
where $F$ and $G$ are symmetric positive semidefinite, and the columns of the $n \times m$ matrix $X$ typically represent data points.
Applications include various dimension reduction methods, such as FDA, LDA, LDE, LPP, ONPP, and OLPP; see, e.g., \cite{shi2020general}.
In the small-sample size case, where the number of sample points is less than the dimension, this pencil is singular.
Yet another (new) source of symmetric singular pencils is described in Section~\ref{sec:appl}.

It is well known that in many situations the $QZ$ algorithm can be used to
compute the eigenvalues of a singular matrix pencil. In addition to values
that are close to the original eigenvalues, the algorithm will also produce additional
values that result from perturbations by rounding errors made to the singular part of
the pencil; see \cite{Wil79} (and also \cite{DTFM_08_1stOrderPencil} for a detailed
first order perturbation analysis of singular pencils). Then the question arises how
to distinguish the ``true'' eigenvalues from the ``random eigenvalues'' among the
values produced by the $QZ$ algorithm. To circumvent this problem, Van Dooren suggested to
first extract the regular part of the pencil \cite{VD79} and this idea was implemented
in Guptri \cite{demmel1993generalized, Guptri} which probably has been the standard
method for solving singular generalized eigenvalue problems for more than four decades.
However Guptri is based on a staircase algorithm that needs rank decisions and may
therefore fail in situations when these decisions become critical.

Recently, alternative methods for the solution of singular eigenvalue problems
have been developed. In \cite{LotzNoferini} it is suggested to directly apply
the QZ algorithm to a singular pencil and separate the true from the fake eigenvalues
with the help of so-called weak condition numbers. In \cite{DanielIvana} a similar
approach is used that involves a small norm perturbation to the pencil before the
$QZ$ algorithm is applied.

The work in this paper follows a different approach and is a sequel to
Part I \cite{HMP19} and Part II \cite{HMP_SingGep2}, which deal with three
methods for nonsymmetric singular pencils: rank-completing perturbations,
rank-projections, and augmentation. All three methods are based on the following
theoretical background developed in \cite{HMP19}:
let the normal rank of the pencil be $r = \text{nrank}(A,B) = \max_{\zeta \in \C} \text{rank}(A+\zeta B)$, define $k = n-r$ and consider a perturbation of the form
\[
\widetilde A-\lambda \widetilde B=(A + \tau \,U D_A V^*)-\lambda \,(B + \tau \,U D_B V^*)
\]
where $\tau\in\mathbb C$, $U,V\in\mathbb C^{n,k}$ and $D_A-\lambda D_B$ is a regular pencil of size $k\times k$. Such perturbations are called \emph{rank-completing},
because the rank of the perturbation pencil $UD_AV^*-\lambda \, UD_BV^*$ is just large
enough to make the pencil $\widetilde A-\lambda\widetilde B$ regular. Generically,
this will be the case and, in addition, the spectrum of
$\widetilde A-\lambda\widetilde B$
consists of three types of eigenvalues: \emph{true eigenvalues} which are
exactly the (finite and infinite) eigenvalues of the original pencil $A-\lambda B$,
\emph{prescribed eigenvalues} which are precisely the eigenvalues of $D_A-\lambda D_B$, and
\emph{random eigenvalues} resulting from perturbation of the singular blocks of
the original pencil $A-\lambda B$. These three types of eigenvalues can be distinguished
with the help of orthogonality relations satisfied by the corresponding eigenvectors.
Computing eigentriplets $(\lambda, x, y)$ of the pencil $\widetilde A-\lambda\widetilde B$ (right and left eigenvectors $x$ and $y$ are well defined due to the regularity of
the pencil), we find that we have $V^*x=U^*y=0$ for true eigenvalues and
$V^*x, U^*y\neq 0$ for prescribed eigenvalues, while for random eigenvalues exactly one
of the quantities $V^*x$ and $U^*y$ is zero; see, e.g.,
\cite[Summary 4.7 and Table~5.1]{HMP19}.

For real symmetric or complex Hermitian pencils, it is beneficial to consider
structure-preserving perturbations to preserve the so-called \emph{sign characteristic} of real eigenvalues. Concerning real symmetric pencils and a
symmetric perturbation of the form $(A + \tau \,U D_A U^T)-\lambda \, (B + \tau \,U D_B U^T)$ with $U \in \R^{n,k}$, the following challenge was observed in Part I \cite[Rem.~4.8]{HMP19}:
\begin{quote}
``If $A$ and $B$ are symmetric, then it seems that for our current approach
we have to use nonsymmetric rank completing perturbations. Namely, when a
symmetric perturbation [\dots] is used, there is an issue with
the third group [\dots] as random eigenvalues appear either as double real eigenvalues or in complex conjugate pairs, and in the former case the orthogonality constraints cannot be satisfied. We leave the study of structured singular pencils for future research.''
\end{quote}

Indeed, the problem is that for symmetric pencils and a symmetric update, it is no longer possible that exactly one of $U^Tx$ and $U^Ty$ is zero, since left eigenvectors for real eigenvalues of symmetric pencils
are also right eigenvectors for the same eigenvalue.
Fortunately, as we show in this paper, it is possible to consider
{\em Hermitian rank-completing perturbations}
$(A + \tau \,U D_A U^*)-\lambda \, ( B + \tau \,U D_B U^*)$, where $\tau\in\mathbb R$, $U \in \C^{n,k}$ and where $D_A-\lambda D_B$ is a regular Hermitian $k\times k$ pencil.

The results on generic rank-completing perturbations from \cite{HMP19}
cannot be applied in this paper, because here we consider structure-preserving
perturbations with the property $V=U$. As a consequence, the meaning of genericity
differs from the one used in \cite{HMP19} and hence we have to develop new theoretical results.

The rest of the paper is organized as follows.
In Section~\ref{sec:prelim} we recall some necessary facts.
Section~\ref{sec:rankcomplete} proposes a rank-completing Hermitian perturbation as a structure-preserving version of the method of Part I \cite{HMP19}.
Section~\ref{sec:proj} discusses an alternative technique, a projection method, as a structure-preserving generalization of one of the two approaches of Part II \cite{HMP_SingGep2}. In Section 5 we discuss for which pencils with other
symmetry structures the results of this paper can be applied as well.
We illustrate the theoretical results with numerical experiments in Section~\ref{sec:appl},
where we also describe an application to bivariate polynomial systems. Finally,
we summarize some conclusions in Section~\ref{sec:concl}.

\section{Preliminaries}
\label{sec:prelim}
First, we define the notion of \emph{genericity} that is used in this paper.
A subset of $\mathbb R^{m}$ is an \emph{algebraic set} if it is the
set of common zeros of finitely many real polynomials in $m$ variables. A set
$\Omega\subseteq\mathbb R^m$ is called \emph{generic} if its complement is contained
in an algebraic set which is not the full space $\mathbb R^m$. (The latter condition
is satisfied if none of the underlying polynomials in $m$ variables is the zero
polynomial.) We call a set $\Omega\subseteq\mathbb C^{n,k}$ \emph{generic} if it
can be identified with a generic subset of $\mathbb R^{2nk}$ in the standard way.
Thus, genericity of a set of complex matrices is interpreted as genericity with
respect to the real and imaginary parts of the entries. This point of view is
important when complex conjugation is involved, because
complex conjugation is not a polynomial map over $\mathbb C$, but it is a polynomial
map over $\mathbb R$ when $\mathbb C$ is canonically identified with $\mathbb R^2$
via real and imaginary parts of complex numbers. Clearly, the set of all matrices
of full rank is a generic subset of $\mathbb C^{n,k}$. Since the intersection of
finitely many generic sets is again generic, we may always assume that a generic
set $\Omega\subseteq\mathbb C^{n,k}$ contains only matrices of full rank.

Next, we recall the well-known Kronecker canonical form (KCF) for matrix pencils; see, e.g., \cite{Gan59}.

\begin{theorem}[Kronecker canonical form]\label{thm:kcf}
Let $A,B\in\mathbb C^{n,m}$. Then there exist nonsingular matrices
$P\in\mathbb C^{n,n}$ and $Q\in\mathbb C^{m,m}$ such that
\[
P\,(A-\lambda B)\,Q=\left[\begin{array}{cc} R(\lambda)&0\\ 0 & S(\lambda)\end{array}\right],
\]
where $R(\lambda)=\operatorname{diag}(J-\lambda I_r,I_s-\lambda N)$ and where $J$ and
$N$ are in Jordan canonical form with $N$ nilpotent and $r,s\geq 0$. Furthermore,
\[
S(\lambda)=\operatorname{diag}\big(L_{m_1}(\lambda),\dots,L_{m_k}(\lambda), \,
L_{n_1}(\lambda)^\top,\dots,L_{n_\ell}(\lambda)^\top\big),
\]
where $L_{\eta}(\lambda)=[0 \ \, I_{\eta}]-\lambda \, [I_{\eta} \ \, 0]$
is of size $\eta\times (\eta+1)$, and $m_i,n_j\ge 0$ for $i=1,\dots,k$ and
$j=1,\dots,\ell$, and where $k,\ell\geq 0$. The canonical form is uniquely
determined up to permutation of blocks in the matrices $J$ and $N$ and in the pencil
$S(\lambda)$.
\end{theorem}

The pencil $R(\lambda)$ is called the regular part of $A-\lambda B$ and contains
the \emph{eigenvalues} of the pencil. The \emph{finite eigenvalues} of $A-\lambda B$
are the eigenvalues of the matrix $J$ while the matrix $N$ represents the
\emph{eigenvalue} $\infty$. The numbers $m_1,\dots,m_k$ and $n_1,\dots,n_\ell$ are
called the \emph{right} and \emph{left minimal indices} of $S(\lambda)$. The values
$m_i=0$ and $n_j=0$ are possible, in this case the pencil $P\,(A-\lambda B)\,Q$ has
a zero column or row, respectively, in the corresponding place.

More important for our setting is the canonical form for Hermitian pencils that is
obtained under structure-preserving transformations, i.e., via congruence. It seems
that a complete and correct canonical form was first presented in \cite{Tho76},
and we therefore refer to it as the \emph{Thompson canonical form} (TCF). To be able
to state it, we need to introduce the following notation. For $\mu\in\mathbb R$ and
$n\in\mathbb N$, we define the real $n\times n$ pencils
\[
Z_{\mu,n}(\lambda):=\smtxa{cccc}{0&\dots&0&\mu\\ \vdots & \adots & \mu & 1\\
0&\adots&\adots & \\ \mu&1&&0}-\lambda
\left[\begin{array}{ccc} 0&&1\\  & \adots & \\ 1&&0\end{array}\right].
\]
and
\[
N_{n}(\lambda):=\left[\begin{array}{ccc} 0&&1\\  & \adots & \\ 1&&0\end{array}\right]-\lambda \ \smtxa{cccc}{0&\dots&\dots&0\\ \vdots &  & \adots & 1\\
\vdots&\adots&\adots & \\ 0&1&&0}.
\]

\begin{theorem}[Thompson canonical form]\label{thm:thompson}
Let $A,B\in \C^{n,n}$ be Hermitian. Then there exist a nonsingular matrix
$P\in \C^{n,n}$ such that
\begin{equation}\label{eq:tnf}
P\,(A-\lambda B)\,P^* =
\left[\begin{array}{cc} R(\lambda)&0\\ 0 & S(\lambda)\end{array}\right],
\quad\mbox{ with }\; R(\lambda)=\left[\begin{array}{ccc}R_{\rm real}(\lambda)&0&0\\
0&R_{\rm comp}(\lambda)&0\\ 0&0&R_{\rm inf}(\lambda)\end{array}\right]
\end{equation}
where
\[
R_{\rm real}(\lambda) = \operatorname{diag}\big(\sigma_1Z_{\mu_1,n_1}(\lambda),\dots,
\sigma_rZ_{\mu_r,n_r}(\lambda)\big),\quad
R_{\rm comp}(\lambda)=\left[\begin{array}{cc}0&J-\lambda I\\ J^*-\lambda I&0
\end{array}\right]
\]
and
\[
R_{\rm inf}(\lambda)=\operatorname{diag}(\sigma_{r+1}N_{n_{r+1}},
\dots,\sigma_{r+s}N_{n_{r+s}})
\]
with $\mu_1,\dots,\mu_r\in\mathbb R$ (not necessarily pairwise distinct),
$J$ a matrix in Jordan canonical form having only eigenvalues with positive
imaginary part, $n_1,\dots,n_{r+s}\in\mathbb N$, and
$\sigma_1,\dots,\sigma_{r+s}\in\{+1,-1\}$.
Furthermore,
\begin{equation} \label{eq:singpart}
S(\lambda)=\operatorname{diag}\left(\left[\begin{array}{cc}0&L_{m_1}(\lambda)\\ L_{m_1}(\lambda)^\top&0\end{array}\right],
\dots,\left[\begin{array}{cc}0&L_{m_k}(\lambda)\\ L_{m_k}(\lambda)^\top&0\end{array}\right]\right),
\end{equation}
where
$L_{j}(\lambda)=[0 \ \,
I_{\eta}]-\lambda \, [I_{\eta} \ \ 0]$
is of size $\eta\times (\eta+1)$, and $m_i\ge 0$ for $i=1,\dots,k$. The form is
unique up to permutation of blocks in $R(\lambda)$ and $S(\lambda)$.
\end{theorem}

The blocks $R_{\rm real}(\lambda)$ and $R_{\rm comp}(\lambda)$
in~\eqref{eq:tnf} contain the real and non-real complex finite eigenvalues
of the Hermitian pencil, while the block $R_{\rm inf}(\lambda)$ contains
blocks associated with the eigenvalue infinity. It follows that the spectrum
of a Hermitian pencil is symmetric with respect to the real line, i.e., if
$\lambda_0\in\mathbb C\setminus\mathbb R$ is an eigenvalue of a Hermitian pencil,
then so is $\overline{\lambda}_0$ with the same multiplicity.
The values $\sigma_1,\dots,\sigma_{r+s}\in\{+1,-1\}$ attached to blocks
associated with real finite eigenvalues or the eigenvalue infinity are
additional invariants of Hermitian pencils under congruence. Their collection
is referred to as the \emph{sign characteristic} of the Hermitian pencil.
For further information on the sign characteristic we refer to \cite{LanR05a}
or to \cite{GohLR05}, where the topic is discussed in terms of related
$H$-self-adjoint matrices. In comparison to the Kronecker canonical form,
we observe that the minimal indices of singular Hermitian pencils occur in
pairs of equal left and right minimal indices. Therefore, when we
say in the following that a Hermitian pencil has the minimal indices $m_1,\dots,m_k$, we mean that it has exactly this list as left minimal indices
and also as right minimal indices.

If $A-\lambda B$ is a singular Hermitian pencil and if a linear combination of
$A$ and $B$ is positive semidefinite, then it follows immediately by inspection from
the Thompson canonical form that all minimal indices are equal to zero. In that case,
singularity of the pencil $A-\lambda B$ is equivalent to $A$ and $B$ having a common
kernel. The challenging case is thus when no linear combination of $A$ and $B$ is
positive semidefinite.

\section{Rank-completing perturbations of Hermitian singular pencils}
\label{sec:rankcomplete}
Let $A,B\in\mathbb C^{n,n}$ such that $A-\lambda B$ is a singular Hermitian pencil
with normal rank $n-k$ for some $k>0$. In this section
we investigate the effect of \emph{structure-preserving rank-completing perturbations},
i.e., generic rank-$k$ perturbations of the form
\begin{equation}
\label{pert1awithtau}
\widetilde A-\lambda \widetilde B := A-\lambda B+\tau \, (UD_AU^*-\lambda \, UD_BU^*),
\end{equation}
where $D_A,D_B\in\mathbb C^{k,k}$ are Hermitian matrices such that $D_A-\lambda D_B$ is regular, $U\in \mathbb C^{n,k}$ has full column rank, and $\tau\in\mathbb R$ is nonzero.

The following result is an extension of \cite[Thm.~4.3]{HMP19} to the Hermitian case.
The proof combines ideas from the proof of \cite[Thm.~4.2]{mehl2017parameter}, where the special
case $k=1$ has been covered for Hermitian pencils,
and from \cite[Thm.~4.3]{HMP19}, where a corresponding result for pencils without additional symmetry structure
has been presented. Note that the latter result cannot be directly applied in
the current context, because it is formulated for {\em strict equivalence} of pencils--a transformation that will not preserve the Hermitian structure in general.
Instead, structure-preserving {\em congruence transformations} should be considered, but this leads to a different concept of genericity.

\begin{proposition}\label{prop:genpert} (Cf.~\cite[Prop.~4.2]{HMP19})
Let $A-\lambda B$ be an $n\times n$ singular Hermitian matrix pencil of normal rank $n-k$.
Then there exists a generic set $\Omega_1\subseteq\mathbb C^{n,k}$
such that for each $U\in\Omega_1$ there exists a nonsingular matrix $P\in\mathbb C^{n,n}$ such that
\[
P\,(A-\lambda B)\,P^*=\left[\begin{array}{cc}R(\lambda)&0\\ 0& S(\lambda)\end{array}\right]
\quad\mbox{and}\quad
PU=\left[\begin{array}{c}0\\ \widetilde U\end{array}\right],
\]
where $R(\lambda)$ and $S(\lambda)$ are the regular and singular parts of $A-\lambda B$,
respectively, and $PU$ is partitioned conformably with $P\,(A-\lambda B)\,P^*$.
\end{proposition}

\begin{proof}
We distinguish the following two cases.

\noindent
{\it Case (1): The pencil $A-\lambda B$ does not have the eigenvalue $\infty$.}

Applying an appropriate congruence otherwise, we may assume without loss of generality that $A-\lambda B$ is in a permuted
version of the Thompson canonical form of Theorem~\ref{thm:thompson}, i.e., we have
\[
A-\lambda B=\left[\begin{array}{ccc}R(\lambda)&0&0\\ 0&0&\mathcal L(\lambda)\\
0&\mathcal L(\lambda)^\top & 0\end{array}\right],\quad U=\left[\begin{array}{c}U_1\\ U_2\\ U_3\end{array}\right],
\]
where $\mathcal L(\lambda)=\operatorname{diag}\big(L_{m_1}(\lambda),\dots, L_{m_k}(\lambda)\big)$ for some values $m_1,\dots,m_k\in\mathbb N$
(which are the minimal indices of $A-\lambda B$), and where $U$ is partitioned conformably with $A-\lambda B$.
Assume that $r$ is the size of $R(\lambda)$, i.e., we have $n=r+2m+k$ for $m=m_1+\dots+m_k$.
We aim to eliminate the block entry $U_1$ in $U$. Applying a congruence transformation with a matrix $\mathcal X$ of the form
\[
\mathcal X=\left[\begin{array}{ccc}I_r&0&X\\ 0&I_m&0\\ 0&0&I_{m+k}\end{array}\right],
\]
where $X\in\mathbb C^{r,m+k}$ will be specified later, we obtain
\[
\mathcal X\,(A-\lambda B)\,\mathcal X^*=\left[\begin{array}{ccc}R(\lambda)&X\mathcal L(\lambda)^\top&0\\
\mathcal L(\lambda)X^*&0&\mathcal L(\lambda)\\ 0&\mathcal L(\lambda)^\top &0\end{array}\right],\quad
\mathcal XU=\left[\begin{array}{c}U_1+XU_3\\ U_2\\ U_3\end{array}\right].
\]
This transformation introduces an unwanted entry in the $(1,2)$- and $(2,1)$-block positions of $\mathcal X\,(A-\lambda B)\,\mathcal X^*$.
We eliminate this block entry by applying a congruence transformation with a matrix $\mathcal Y$ of the form
\[
\mathcal Y=\left[\begin{array}{ccc}I_r&0&0\\ Y^*&I_m&0\\ 0&0&I_{m-k}\end{array}\right],
\]
where, again, $Y\in\mathbb C^{m,r}$ will be specified later. This yields
\[
\mathcal Y\mathcal X\,(A-\lambda B)\,\mathcal X^*\mathcal Y^*=\left[\begin{array}{ccc}R(\lambda)&X\mathcal L(\lambda)^\top+R(\lambda)Y&0\\
Y^*R(\lambda)+ \mathcal L(\lambda)X^*& M(\lambda)&\mathcal L(\lambda)\\ 0 &\mathcal L(\lambda)^\top &0\end{array}\right]\]
and
\[
\mathcal Y\mathcal XU=\left[\begin{array}{c}U_1+XU_3\\ \widetilde U_2\\ U_3\end{array}\right],
\]
where $M(\lambda)=Y^*X\mathcal L(\lambda)^\top+L(\lambda)X^*Y$ and $\widetilde U_2=U_2+Y^*U_1+Y^*XU_3$.
Our aim is to choose $X$ and $Y$ as solutions to the equations
\begin{equation}\label{eq:pencilmatrix}
X\mathcal L(\lambda)^\top+R(\lambda)\,Y = 0\quad\mbox{and}\quad U_1+XU_3=0.
\end{equation}
Let $A_r$ and $B_r$ be the coefficients of $R(\lambda)$ and define
\[
G:=\left[\begin{array}{ccccccc}0&I_{m_1}&&&&&\\ &&0&I_{m_2}&&&\\ &&&&\ddots &&\\ &&&&&0&I_{m_k}\end{array}\right],\
H:=\left[\begin{array}{ccccccc}I_{m_1}&0&&&&&\\ &&I_{m_2}&0&&&\\ &&&&\ddots &&\\ &&&&&I_{m_k}&0\end{array}\right],
\]
i.e., we have $R(\lambda)=A_r-\lambda B_r$ and $\mathcal L(\lambda)=G-\lambda H$. Then from~\eqref{eq:pencilmatrix}
we obtain  three matrix equations
\[
XG^\top+A_rY=0,\quad X^*H^\top+B_rY=0,\quad U_1+XU_3=0.
\]
Since $\infty$ is not an eigenvalue of $A-\lambda B$, the matrix $B_r$ is invertible and hence the second
equation can be solved for $Y$ yielding $Y=-B_r^{-1}XH^\top$. Inserting this into the first equation and multiplying it
with $B_r^{-1}$ from the left yields
\begin{equation}\label{7.1.23}
B_r^{-1}XG^\top=B_r^{-1}AB_r^{-1}XH^\top.
\end{equation}
Writing $B_r^{-1}X=[\vecx_1^{(1)} \, \dots \ \vecx_{m_1+1}^{(1)} \ \ \dots \ \ \vecx_1^{(k)}\, \dots \ \vecx_{m_k+1}^{(k)}]$
and comparing the columns on both sides of~\eqref{7.1.23}, we obtain
\begin{align*}
B_r^{-1}X & =[\,\vecx_1^{(1)} \; \ B_r^{-1}A_r\vecx_1^{(1)}\,\dots\, \ (B_r^{-1}A_r)^{m_1}\vecx_{1}^{(1)} \; \dots \\
 & \phantom{MMMM} \vecx_1^{(k)} \;\; B_r^{-1}A_r\vecx_1^{(k)}\,\dots\, \ (B_r^{-1}A_r)^{m_k}\vecx_{1}^{(k)}\,].
\end{align*}
It remains to fix the columns $\vecx_1^{(i)}$, $i=1,\dots,k$, such that also the third equation $U_1+XU_3=0$, or, equivalently,
$B_r^{-1}U_1=-B_r^{-1}XU_3=0$ is satisfied. To this end, let us use the partitions
\[
U_3=\left[\begin{array}{c}U_3^{(1)}\\[-1mm] \vdots\\ U_3^{(k)}\end{array}\right],\quad
U_3^{(i)}=\left[\begin{array}{c}(\vecu_1^{(i)})^\top\\[-1mm] \vdots\\ (\vecu_{m_i+1}^{(i)})^\top\end{array}\right],\;i=1,\dots,k.
\]
Then, using the identity $\operatorname{vec}(XYZ)=(Z^\top \otimes X)\operatorname{vec}(Y)$
for the vectorization operator and the Kronecker product, we obtain
\begin{eqnarray*}
\operatorname{vec}(B_r^{-1}U_1)&=&\operatorname{vec}(B_r^{-1}XU_3)
=\operatorname{vec}\Big(\sum_{i=1}^k\sum_{j=1}^{m_i+1}(B_r^{-1}A_r)^{j-1}\vecx_1^{(i)}(\vecu_j^{(i)})^\top\Big)\\
&=&\sum_{i=1}^k\sum_{j=1}^{m_i+1} \big(\vecu_j^{(i)}\otimes (B_r^{-1}A_r)^{j-1}\big)\,\vecx_1^{(i)}
=M\cdot\left[\begin{array}{c}\vecx_1^{(1)}\\[-1mm] \vdots\\ \vecx_1^{(k)}\end{array}\right],
\end{eqnarray*}
where
\[
M=\Bigg[\sum_{j=1}^{m_1+1}\big(\vecu_j^{(1)}\otimes (B_r^{-1}A_r)^{j-1}\big) \ \dots \
\sum_{j=1}^{m_k+1}\big(\vecu_j^{(k)}\otimes (B_r^{-1}A_r)^{j-1}\big)\Bigg]
\]
is an $rk\times rk$ matrix whose determinant is a polynomial in the real and imaginary parts of the entries of $U_3$.
This polynomial is nonzero, because the particular choice $\vecu_1^{(i)}=e_i$ (the $i$th canonical basis vector) and $\vecu_j^{(i)}=0$ for $j=2,\dots,m_i+1$, $i=1,\dots,k$,
yields $M=I_{rk}$. Thus, generically with respect to the real and imaginary parts of the entries of the matrix $U_3$, the
columns $\vecx_1^{(1)},\dots,\vecx_1^{(k)}$ can be chosen such that $U_1+XU_3=0$ and $X\mathcal L(\lambda)^\top+R(\lambda)Y=0$
for
\begin{align*}
X & = B_r\ [\,\vecx_1^{(1)} \ \ B_r^{-1}A_r\vecx_1^{(1)}\dots \ (B_r^{-1}A_r)^{m_1}\vecx_{1}^{(1)} \ \dots \\
& \phantom{MMMMMMMM} \ \vecx_1^{(k)} \ \ B_r^{-1}A_r\vecx_1^{(k)}\dots \ (B_r^{-1}A_r)^{m_k}\vecx_{1}^{(k)}\,]
\end{align*}
and $Y=-B_r^{-1}XH^\top$. Thus, we have shown that for these choices of $X$ and $Y$ we have
\[
\mathcal Y\,\mathcal X\,(A-\lambda B)\,\mathcal X^*\,\mathcal Y^*=\left[\begin{array}{ccc}R(\lambda)&0&0\\
0& M(\lambda)&\mathcal L(\lambda)\\0 &\mathcal L(\lambda)^\top &0\end{array}\right]\quad\mbox{and}\quad
\mathcal Y\mathcal XU=\left[\begin{array}{c}0\\ \widetilde U_2\\ U_3\end{array}\right].
\]
Since $A-\lambda B$ and $\mathcal Y\,\mathcal X\,(A-\lambda B)\,\mathcal X^*\,\mathcal Y^*$ are congruent, its singular
parts must be congruent as well, so applying an appropriate congruence on the singular part of
$\mathcal Y\,\mathcal X\,(A-\lambda B)\,\mathcal X^*\,\mathcal Y^*$ (note that the corresponding transformation will not
destroy the zero in the first block entry of $\mathcal Y\mathcal XU$), we finally have proven the statement of the theorem.

\noindent
{\it Case (2): The pencil $A-\lambda B$ does have the eigenvalue $\infty$.}
This case can be reduced to the first case with the help of a M\"obius transformation.
Indeed, choose $\zeta_1$, $\zeta_2\in\mathbb R$ such that $\zeta_1^2+\zeta_2^2=1$ and the pencil
\[
\mathrm M_{\zeta_1,\zeta_2}(A-\lambda B):=\zeta_1 A+\zeta_2 B-\lambda \, (\zeta_1 B-\zeta_2 A)
\]
does not have the eigenvalue $\infty$. This is possible, because the considered M\"obius transformation
does not change the Kronecker structure of the pencil except for rotating the spectrum around the
origin with a specific angle; see \cite{MMMM15}. Furthermore, the transformed pencil is still
Hermitian, as $\zeta_1$ and $\zeta_2$ are taken to be real. Then using the result from Case~1 on this
pencil and then applying the inverse M\"obius transformation
\[
\mathrm M_{\zeta_1,-\zeta_2}(C-\lambda D):=\zeta_1 C-\zeta_2 D-\lambda \, (\zeta_1 D+\zeta_2 C)
\]
proves the result under the presence of the eigenvalue $\infty$.
\end{proof}

\begin{theorem}\label{thm:nfp} (Cf.~\cite[Thm.~4.3]{HMP19})
Let $A-\lambda B$ be an $n\times n$ singular Hermitian pencil of normal rank $n-k$
and let $D_A,D_B\in\mathbb C^{k,k}$ be such that $D_A-\lambda D_B$ is regular and Hermitian and such that all
eigenvalues of $D_A-\lambda D_B$ are distinct from the eigenvalues of $A-\lambda B$.
Then there exists a generic set $\Omega\subseteq\mathbb C^{n,k}$ such that for all
$U\in\Omega$ the following statements hold for the perturbed pencil
$\widetilde A-\lambda\widetilde B:=A-\lambda B+\tau \, (UD_AU^*-\lambda \, UD_BU^*)$:
\begin{enumerate}
\item There exists a nonsingular matrix $P$ such that for each $\tau\in\mathbb R\setminus\{0\}$ we have
\begin{equation}\label{eq:nfp}
P \,(\widetilde A-\lambda\widetilde B)\, P^*=\left[\begin{array}{cc}R(\lambda)&0\\
0& R_{\rm new}(\lambda)\end{array}\right],
\end{equation}
where $R(\lambda)$ is the regular part of the original pencil $A-\lambda B$,
and $R_{\rm new}(\lambda)$ is regular and all its eigenvalues are distinct from the eigenvalues of $R(\lambda)$.
\item If $\lambda_0$ is an eigenvalue of $A-\lambda B$, i.e., $\mathrm{rank}(A-\lambda_0 B)<n-k$ if $\lambda_0\ne\infty$ or
$\mathrm{rank}(B)<n-k$ if $\lambda_0=\infty$,
then $\lambda_0$ is an eigenvalue of~\eqref{pert1awithtau} for each $\tau\in\mathbb R\setminus\{0\}$. Furthermore,
if $(x_1,\dots,x_\ell)$ is a left or right Jordan chain of $\widetilde A-\lambda \widetilde B$ associated with $\lambda_0$ for some $\tau\in\mathbb R\setminus\{0\}$,
then it is also a left resp. right Jordan chain of $\widetilde A-\lambda \widetilde B$ associated with $\lambda_0$ for any $\tau\in\mathbb R\setminus\{0\}$.
In addition, $U^*x_i=0$ for $i=1,\dots,\ell$.
\end{enumerate}
\end{theorem}

\begin{remark} \rm
With respect to the second point, we note that a left Jordan chain
associated with an eigenvalue $\lambda_0$ is also a right Jordan chain
associated with the eigenvalue $\overline\lambda_0$. In particular, if $\lambda_0$ is real then left and right Jordan chains coincide.
\end{remark}

\begin{proof}
As in the proof of Proposition~\ref{prop:genpert} we will distinguish two cases.

\noindent
{\it Case (A): The pencil $A-\lambda B$ does not have the eigenvalues $\infty$ or zero.}

Let $\tau\in\mathbb R\setminus\{0\}$ be fixed. Applying Proposition~\ref{prop:genpert},
there generically exists a nonsingular matrix $P$ such that
\begin{equation}\label{eq:help1}
P\,(A-\lambda B)\,P^*=\left[\begin{array}{cc}R(\lambda)&0\\ 0& S(\lambda)\end{array}\right],\quad
PU=\left[\begin{array}{c}0\\ U_2\end{array}\right],
\end{equation}
where $R(\lambda)$ and $S(\lambda)$ are the regular and singular parts of $A-\lambda B$,
respectively, and where $PU$ is partitioned conformably with $P\,(A-\lambda B)\,P^*$.

We first show 1. Let $m_1,\dots,m_k$ be the minimal indices of $A-\lambda B$. We can then assume that $S(\lambda)$ is in the
canonical form of \eqref{eq:singpart}.
Then the perturbed pencil $\widetilde A-\lambda\widetilde B$ satisfies
\[
P\,(\widetilde A-\lambda \widetilde B)\,P = \left[\begin{array}{cc}R(\lambda) & 0\\
0 & R_{\rm new}(\lambda)\end{array}\right],
\]
with $R_{\rm new}(\lambda):= S(\lambda)+\tau \, U_2\,(D_A-\lambda D_B)\,U_2^*$.

Now let $\lambda_0\in\mathbb C$ be an eigenvalue of $R(\lambda)$ of geometric
multiplicity $\ell$, i.e., in particular we have $\det R(\lambda_0)=0$.
We show that there exists a generic set $\Omega_2$ such that for all $U\in\Omega_2$ we have rank$(\widetilde A-\lambda_0 \widetilde B)\ge n-\ell$.
Then we can conclude that for all $U\in\Omega_1\cap\Omega_2$ the corresponding pencil $R_{\rm new}(\lambda)$ is regular and
does not have $\lambda_0$ as an eigenvalue.

To see this, let $Q\in\mathbb C^{k,k}$ be unitary such that $Q\,(D_A-\lambda_0D_B)\,Q^*$ is in Schur form with diagonal elements
$\alpha_1,\dots,\alpha_k\in\mathbb C$ in that order. (Observe that since $\lambda_0$ may be complex, the matrix $D_A-\lambda_0D_B$ is not necessarily Hermitian and so its
eigenvalues may be complex as well.)
Since by hypothesis $\lambda_0$ is not an eigenvalue of $D_A-\lambda D_B$, we have that
$D_A-\lambda_0D_B$ has full rank and hence $\alpha_1,\dots,\alpha_k$ are nonzero. Replacing $U_2$ with
\[
\widetilde U_2=[\,e_{m_1+1} \ \ e_{2m_1+m_2+2} \ \dots \ e_{2m_1+\cdots+2m_{k-1}+m_k+k}\,] \cdot Q
\]
we obtain that the matrix $S(\lambda_0)+\tau \, \widetilde U_2\,(D_A-\lambda_0 D_B)\,\widetilde U_2^*$ is block upper triangular
with diagonal blocks of the form
\[
\left[\begin{array}{cc}0&L_{m_i}(\lambda_0)\\ L_{m_i}(\lambda_0)^\top&
{\scriptsize\left[\begin{array}{cc}\alpha_i&0\\ 0&0\end{array}\right]}\end{array}\right].
\]
It follows that
\[
\det\big(S(\lambda_0)+\tau \, \widetilde U_2\,(D_A-\lambda_0 D_B)\,\widetilde U_2^*\big)
=(-1)^{m_1+\cdots+m_k}\lambda_0^{2m_1+\cdots+2m_k}\cdot\prod_{i=1}^k\alpha_i
\]
which is nonzero as $\lambda_0,\alpha_1,\dots,\alpha_k$ are all nonzero. Thus, we have shown that
for a particular choice $U\in\mathbb C^{n,k}$ (and all $\tau\ne 0$) we have rank$(\widetilde A-\lambda_0 \widetilde B)=n-\ell$ for the
corresponding perturbed pencil. Then by \cite[Lemma 2.1]{MehMRR14} the set $\Omega_2$ of all $U\in\mathbb C^{n,k}$
for which rank$(\widetilde A-\lambda_0 \widetilde B)\ge n-\ell$ is a generic set.
By induction on the number of pairwise distinct eigenvalues of $A-\lambda B$, we arrive at a generic set $\Omega$
with the stated properties in 1) by using that the intersection of finitely many generic sets is still generic.

Next, we show 2. Let $\lambda_0\in\mathbb C$ be an eigenvalue of $A-\lambda B$. As this means that $\lambda_0$ is then an eigenvalue
of the regular part $R(\lambda)$, it follows immediately from 1) that $\lambda_0$ is an eigenvalue
of $\widetilde A-\lambda\widetilde B$ for any $\tau\ne 0$. Now let $\tau\ne 0$ be fixed and let
$(x_1,\dots,x_\nu)$ be a right Jordan chain of $\widetilde A-\lambda\widetilde B$ associated with $\lambda_0$,
i.e., we have
\[
(\widetilde A-\lambda_0\widetilde B)\,x_1=0\quad\mbox{and}\quad (\widetilde A-\lambda_0\widetilde B)\,x_i=\widetilde Bx_{i-1},\quad i=2,\dots,\nu.
\]
Partition
\[
P^{-*}x_i=\left[\begin{array}{c}x_i^{(1)}\\ x_i^{(2)}\end{array}\right]
\]
conformably with the partition in~\eqref{eq:help1}. Since $R_{\rm new}(\lambda_0)$ is invertible,
it easily follows by induction that $x_i^{(2)}=0$ for $i=1,\dots,\nu$. But then $(x_1,\dots,x_\nu)$
is a right Jordan chain of $\widetilde A-\lambda\widetilde B$ associated with $\lambda_0$ for any $\tau\ne 0$.
In particular, due to the special structure of $PU$ and $P^{-*}x_i$ we immediately obtain $U^*x_i=0$
for $i=1,\dots,\nu$. The statement for left Jordan chains is obtained completely analogously.

{\it Case (B): $A-\lambda B$ does have one (or both) of the eigenvalues $\infty$ or $0$.}
This case is proven with the help of a M\"obius transformation as in the proof of Proposition~\ref{prop:genpert}.
\end{proof}

\begin{lemma}\label{lem:Lo} Let $A-\lambda B$ be an $n\times n$ singular Hermitian pencil of normal rank $n-k$
with minimal indices $m_1,\dots,m_k$ and set $M:=m_1+\cdots+m_k$. Let
$\gamma_1,\dots,\gamma_k\in\mathbb C$ be given values that are distinct from the eigenvalues of $A-\lambda B$.
Then there exists a generic set $\Omega\subseteq\mathbb C^{n,k}$ such that for all $U\in\Omega$ the following statements hold:
\begin{enumerate}
\item There exist exactly $M$ pairwise distinct values $\alpha_1,\dots,
\alpha_M\in\mathbb C$ different from the eigenvalues of $A-\lambda B$ such that for each
$\alpha_i$ there exists a nonzero vector $z_i$ with $(A-\alpha_iB)\,z_i=0$ and $U^*z_i=0$.
Furthermore, these values are different from the values $\gamma_1,\dots,\gamma_k$ and also different from each of the
conjugates $\overline{\alpha}_1,\dots,\overline{\alpha}_M$, i.e., we have
$\{\alpha_1,\dots,\alpha_M\}\cap\{\overline{\alpha}_1,\dots,\overline{\alpha}_M\}=\emptyset.$
\item For any given set of
$k$ linearly independent vectors $t_1,\dots,t_k\in\mathbb C^k$ there exist nonzero vectors
$s_1,\dots,s_k$ with $(A-\gamma_i B)\,s_i=0$ and $t_i=U^*s_i$ for $i=1,\dots,k$.
\end{enumerate}
\end{lemma}

\begin{proof}
1) Up to the first part of the ``furthermore'' part, the statement has already been proven in \cite[Lemma 4.5]{HMP19} even for a
general singular pencil of normal rank $n-k$ without additional structure. However, we need to repeat this proof for the special situation of a Hermitian pencil
to be able to extend the proof to cover the second statement of the furthermore part.

Without loss of generality we may assume that $A-\lambda B$ is in the permuted version
\[
A-\lambda B=\left[\begin{array}{ccc}0&\mathcal L(\lambda)^\top &0\\ \mathcal L(\lambda) & 0 & 0\\ 0&0&R(\lambda)\end{array}\right],\quad
\mathcal L(\lambda)=\operatorname{diag}\big(L_{m_1}(\lambda),\dots, L_{m_k}(\lambda)\big)
\]
of the Thompson canonical form from Theorem \ref{thm:thompson} with $R(\lambda)$ standing for the regular part of $A-\lambda B$.

Let $\alpha\in\mathbb C$ be different from the eigenvalues of $A-\lambda B$. Then the kernel of $A-\alpha B$ has dimension $k$
and a basis is given by the vectors $q_1(\alpha),\dots,q_k(\alpha)$, where
\[
q_j(\alpha)=\left[\begin{array}{c}0\\\hline 1\\[-0.5mm] \alpha\\[-1.5mm] \vdots\\ \alpha^{m_j}\\\hline 0\end{array}\right],
\]
where the zero in the first block entry of $q_j(\alpha)$ is a vector having the length $m_1+\cdots+m_{j-1}+j-1$.

If $z\ne 0$ is a vector satisfying $U^*z=0$ and $(A-\alpha B)\,z=0$, then the second equality implies that has the form
\[z=c_1\,q_1(\alpha)+\cdots+c_k\,q_k(\alpha),\]
where $c=[c_1 \ \, \dots \ \, c_k]^T\ne 0$. The equation $U^*z=0$ can then be rewritten in the form
\begin{equation}\label{eq:gc}
Q(\alpha)\,c=0,
\end{equation}
with a $k\times k$ matrix $Q(\alpha)$ whose $(i,j)$-entry $u_i^*q_j(\alpha)$ is a polynomial in
$\alpha$ having coefficients that are itself polynomials in the real and imaginary parts of the entries of $U$.
It is easily seen (e.g., by choosing the $m_j$-entry of $u_i$ to be nonzero) that this polynomial is of degree
$m_j$ in $\alpha$, generically. Equation \eqref{eq:gc} has a nontrivial solution if and only if
$\det Q(\alpha)=0$. Clearly $\det Q(\alpha)$ is a polynomial in $\alpha$ which generically (with respect to the
real and imaginary parts of the entries of $U$) is of degree $M$.
It follows that $\det Q(\alpha)$ will have exactly $M$ roots $\alpha_1,\dots,\alpha_M$
(counted with multiplicities).

Next, let $\mu\in\mathbb C$ be fixed. Then $\det Q(\mu)$ is a polynomial in the real and imaginary parts of the entries of $U$.
This polynomial is nonzero, because for the particular choice
\[
u_1=e_1, \quad u_2=e_{m_1+2}, \quad \dots, \quad u_k=e_{m_1+\cdots+m_{k-1}+k}
\]
the matrix $Q(\mu)$ is just the $k\times k$ identity. Hence generically with respect to the real and imaginary parts of the
entries of $U$ the value $\mu$ will be distinct from all roots of $Q(\alpha)$ seen as a polynomial in $\alpha$.
Using the fact that the intersection of finitely many generic sets is still generic, we can conclude that the
roots $\alpha_1,\dots,\alpha_m$ of $Q(\alpha)$ will generically be different from finitely many given values
$\gamma_1,\dots,\gamma_k$. (Using again a M\"obius transformation argument, one can show that
also the case is covered that one of the values equals $\infty$.)

To show that the roots $\alpha_1,\dots,\alpha_M$ are generically simple, let us denote $p(\alpha):=\det Q(\alpha)$.
It is well known that the roots of a polynomial $p$ are simple if and only if the determinant of the Sylvester matrix $S(p,p')$ associated
with $p$ and its derivative $p'$ is nonzero.
It is clear from the construction of the Sylvester matrix, that
$\det S(p,p')$ is a polynomial in the real and imaginary parts of the entries of $U$. We have to show that this polynomial is not
the zero polynomial and for this we have to find a particular choice for the entries of $U$ for which the corresponding
roots $\alpha_1,\dots,\alpha_M$ of $p(\alpha)$ are pairwise distinct. To this end, select
$u_1=e_{m_1+1}-\eps_1 e_1$, $u_2=e_{m_1+m_2+2}-\eps_2 e_{m_1+2}$, \dots,
$u_k=e_{m_1+\cdots+m_k+k}-\eps_ke_{m_1+\cdots+m_{k-1}+k}$,
with some parameters $\eps_1,\dots,\eps_k>0$. It follows that that $u_i^*q_j(\alpha)=\delta_{ij}\alpha^{m_j}-\eps_j$
and hence $Q(\alpha)$ is a diagonal matrix with determinant
\[
p(\alpha)=\prod_{j=1}^k \ (\alpha^{m_j}-\eps_j).
\]
The roots of each individual factor $\alpha^{m_j}-\eps_j$ are $m_j$ pairwise distinct complex numbers
lying on a circle centered at zero and with radius $\eps_j^{1/m_j}$. Clearly, we can choose the parameters
$\eps_1,\dots,\eps_k$ such that the $k$ radii $\eps_j^{1/m_j}$ are pairwise distinct.
Consequently, $\det S(p,p')$ is not the zero polynomial and hence generically with respect to the real and imaginary parts of the
entries of $U$ the roots $\alpha_1,\dots,\alpha_M$ are pairwise distinct.

It remains to show that the values $\alpha_1,\dots,\alpha_M$ are also different from each of their
conjugates $\overline{\alpha}_1,\dots,\overline{\alpha}_M$. This is exactly the case when the determinant of the Sylvester matrix
$S(p,\overline{p})$ is nonzero, where $\overline{p}$ is obtained from $p$ by viewing it as a polynomial in $\alpha$ and replacing
all its coefficients by their conjugates. Again, it is clear that the determinant of $S(p,\overline{p})$ is a polynomial in the
real and imaginary parts of the entries of $U$ and it remains to show that this polynomial is nonzero by finding one particular
example for which the values $\alpha_1,\dots,\alpha_M$ are different from each of their conjugates. To this end, we can use the
same example as in the previous argument now choosing $\eps_i$ to be complex with appropriate absolute values and arguments.

2) Since the values $\gamma_1,\dots,\gamma_k$ are different from the eigenvalues of $A-\lambda B$, we have $(A-\gamma_iB)\,s_i=0$
for some $s_i\ne 0$ if there exists constants $c_1,\dots,c_k$ not all  zero such that
\[
s_i=c_1\,q_1(\gamma_i)+\cdots+c_k\,q_k(\gamma_i)
\]
with $q_1,\dots,q_k$ as in 1). Then the equation $U^*s_i=t_i$ translates to
$Q(\gamma_i)\,c=t_i$ for $i=1,\dots,k$ with $Q$ and $c$ as in 1). Since $\gamma_i$ is different from the
values $\alpha_1,\dots,\alpha_M$, we have $\det Q(\gamma_i)\ne 0$ which means that the equation $Q(\gamma_i)\,c=t_i$
has a unique solution for $c$ for $i=1,\dots,k$.
\end{proof}

\begin{theorem}\label{thm:mainBor} Let $A-\lambda B$ be an $n\times n$ singular Hermitian pencil of normal rank $n-k$
with minimal indices $m_1,\dots,m_k$ and let $M:=m_1+\cdots+m_k$. Furthermore, let $D_A,D_B\in\mathbb C^{k,k}$ be Hermitian such that
$D_A-\lambda D_B$ is regular and such that all its eigenvalues
$\gamma_1,\dots,\gamma_k\in\mathbb C$ are semisimple and
different from the eigenvalues of $A-\lambda B$. Then there exists a generic set $\Omega\subseteq\mathbb C^{n,k}$ such that
for all $U\in\Omega$ the following statements hold:
\begin{enumerate}
\item The pencil~\eqref{pert1awithtau} has $2M$ simple eigenvalues $\alpha_1,\dots,\alpha_M,\overline{\alpha}_1,\dots,\overline{\alpha}_M$
which are independent of $\tau\ne 0$. For each eigenvalue $\alpha_i$ its corresponding right eigenvector $x_i$ is constant in $\tau\ne 0$
(up to scaling) and satisfies $U^* x_i=0$, while its corresponding left eigenvector $y_i$ is a linear function of $\tau$ (up to scaling) and
satisfies $U^*y_i\ne 0$ for all $\tau\ne 0$. Furthermore, $x_i$ is also a left eigenvector and $y_i$ a right eigenvector of the
pencil~\eqref{pert1awithtau} associated with $\overline{\alpha}_i$.
\item For each $\tau\ne 0$ each $\gamma_i$ is an eigenvalue of~\eqref{pert1awithtau} with the same algebraic
multiplicity as for the pencil $D_A-\lambda D_B$. Furthermore, the left and right null spaces
${\cal N}_l(\gamma_i)$ and ${\cal N}_r(\gamma_i)$ of~\eqref{pert1awithtau} associated with $\gamma_i$
are constant in $\tau$. In addition, we have ${\cal N}_r(\gamma_i)\cap \ker(U^*)=\{0\}$ and ${\cal N}_l(\gamma_i)\cap \ker(U^*)=\{0\}$, i.e., for each pair $(x,y)$ of right and left eigenvectors $x$ and $y$ of~\eqref{pert1awithtau}
associated with $\gamma_i$ we have $U^*x\ne 0$ and $U^*y\ne 0$.
\end{enumerate}
\end{theorem}
\begin{remark}\rm
The simplicity of the eigenvalues $\alpha_1,\dots,\alpha_M,\overline{\alpha}_1,\dots,\overline{\alpha}_M$ implies
that they are all nonreal, different from the eigenvalues of $A-\lambda B$, and different from the values $\gamma_1,\dots,\gamma_k$.
\end{remark}
\begin{proof}
Applying a M\"obius transformation simultaneously to both $A-\lambda B$ and $D_A-\lambda D_B$ otherwise, we may assume that $\infty$ is not an eigenvalue of $D_A-\lambda D_B$.

Observe that generically with respect to the real and imaginary parts of the entries of $U$, the statements of Lemma~\ref{lem:Lo}
hold, if we take the eigenvalues of the pencil $D_A-\lambda D_B$
for the values $\gamma_1,\dots,\gamma_k$ and a corresponding linearely independent set of eigenvectors $t_1,\dots,t_k$. We now show 1)--3).

1) Applying Lemma~\ref{lem:Lo}, generically with respect to the real and imaginary parts of the entries of $U$ there exist
exactly $M$ pairwise distinct values $\alpha_1,\dots,\alpha_M$ different from the eigenvalues of $A-\lambda B$
and from $\gamma_1,\dots,\gamma_k$ satisfying $\{\alpha_1,\dots,\alpha_M\}\cap\{\overline{\alpha}_1,\dots,\overline{\alpha}_M\}=\emptyset$,
as well as corresponding vectors $z_1,\dots,z_M\ne 0$ satisfying $(A-\alpha_i B)\,z_i=0$ and $U^*z_i=0$ for $i=1,\dots,M$.
It follows that
$$(A-\alpha_i B + \tau \, UD_AU^* -\alpha_i \, \tau \, UD_BU^*)\,z_i=0$$ which implies that
$\alpha_i$ is an eigenvalue of \eqref{pert1awithtau} with right eigenvector $z_i$ for all $\tau\ne 0$.

Now let us interpret $\tau$ as a variable. Then the (non-Hermitian) pencil
\begin{equation}\label{eq:pengh}
G_i+\tau H_i:=(A-\alpha_i B) +\tau \, U(D_A-\alpha_iD_B)\,U^*
\end{equation}
is a singular pencil since $\alpha_i$ is an eigenvalue of~\eqref{pert1awithtau} for all $\tau\in\mathbb C$.
Let us assume that the pencil has $j$ right and $j$ left minimal indices. (Since the pencil is square, the numbers of left and right
minimal indices must be equal.) Since $G_iz_i=H_iz_i=0$, we know that one of the right minimal indices is zero. Since $\alpha_i$ is a
simple eigenvalue of~\eqref{pert1awithtau}, it follows that all other $j-1$ minimal indices are larger than zero (if there are any).
On the other hand all left minimal indices of the pencil~\eqref{eq:pengh} are larger than zero, because otherwise there would exists
a nonzero vector $w_i$ satisfying $w_i^*G_i=0$ and $w_i^*H_i=0$. Then $w_i^*G_i=0$ implies
$w_i^*(A-\alpha_i B)=0$ or, equivalently, $(A-\overline{\alpha}_iB)\,w_i=0$, because $A$ and $B$ are Hermitian.
On the other hand, $w_i^*H_i=0$ implies $w_i^*U=0$,
because $U^*$ generically has full row rank and $D_A-\alpha_i D_B$ is nonsingular since
$\alpha_i$ is different from the eigenvalues of $D_A-\lambda D_B$. But then by Lemma~\ref{lem:Lo} $\overline{\alpha}_i$
is equal to one of the values $\alpha_1,\dots,\alpha_M$ which is a contradiction. Thus, all left minimal
indices of~\eqref{eq:pengh} are at least one.

Moreover, we have $\text{rank}(G_i)=n-k$, because $\alpha_i$ differs from all finite eigenvalues of $A-\lambda B$, and
we have $\text{rank}(H_i)=k$, because $\alpha_i$ differs from the eigenvalues of $D_A-\lambda D_B$.
Thus, the Kronecker canonical form of~\eqref{eq:pengh} contains at least $n-k-j$ blocks associated with the eigenvalue infinity
and at least $k-j$ blocks associated with the eigenvalue zero. Taking into account
that the pencil has $j$ left and $j$ right minimal indices, where we know that one
right minimal index is zero and the others are at least one, it follows that the
Kronecker canonical form of~\eqref{eq:pengh} has at least
$n-k-j+k-j+j\cdot 1+1\cdot 1+(j-1)\cdot 2=n+j-1$ columns
which implies that we must have $j=1$, i.e., the pencil~\eqref{eq:pengh}
has precisely one right minimal index (which we know is zero) and consequently exactly one left minimal index, say $\eta$.
Another standard calculation shows that the Kronecker canonical form of~\eqref{eq:pengh} has at least $n+\eta-1$ rows which implies $\eta=1$.
But then, there exist two linearly independent vectors $w_i$ and $z_i$ satisfying
\begin{align*}
w_i^*(A-\alpha_iB)&=0,\\
z_i^*(A-\alpha_iB)+w_i^*U(D_A-\alpha_i D_B)\,U^*&=0,\\
z_i^*U(D_A-\alpha_i D_B)\,U^*&=0,
\end{align*}
where, in particular, $w_i^*U\ne 0$. It follows that a left eigenvector $y_i$ of the pencil~\eqref{pert1awithtau}
associated with $\alpha_i$ then has (up to scaling) the form $y_i(\tau)=w_i+\tau z_i$. This finishes the proof of 1).

2) As mentioned in the beginning of the proof, let $t_1,\dots,t_k$ form a set of
linearly independent eigenvectors of the pencil $D_A-\lambda D_B$
associated with the eigenvalues $\gamma_1,\dots,\gamma_k$. By Lemma~\ref{lem:Lo}, there exist
$k$ (necessarily linearly independent) vectors $s_1,\dots,s_k\in\mathbb C^n$ such that
$(A-\gamma_i B)\,s_i=0$ and $t_i=U^*s_i$ for $i=1,\dots,k$. Then we have
\[
(\widetilde A-\gamma_i\widetilde B)\,s_i=(A-\gamma_i B)\,s_i+\tau \, U(D_A-\gamma_i D_B)\,U^*s_i=0
\]
for each $\tau\ne 0$ for $i=1,\dots,k$.
This implies that the values $\gamma_1,\dots,\gamma_k$ are eigenvalues of $\widetilde A-\lambda\widetilde B$
with the same algebraic multiplicities as for $D_A-\lambda D_B$. Furthermore, it follows that the null space
${\cal N}_r(\gamma_i)$ does not depend on $\tau$ and by construction we have
${\cal N}_r(\gamma_i)\cap \ker(U^*)=\{0\}$.

By applying Lemma~\ref{lem:Lo} to the pencil $A^*-\lambda B^*$ we obtain the analogous statements for
the left null spaces ${\cal N}_l(\gamma_i)$.
\end{proof}

\begin{summary} \rm\label{rm:3groups}
We summarize the results from Theorem~\ref{thm:nfp} and Theorem~\ref{thm:mainBor}: let
$A-\lambda B$ be an $n\times n$ singular Hermitian pencil of normal rank $n-k$
with left minimal indices $m_1,\dots,m_k$ and let $D_A$, $D_B$, and $M$ be as in Theorem~\ref{thm:mainBor}.
Select a random $U\in\mathbb C^{n,k}$. Then generically $U$ has full rank and satisfies the hypotheses of Theorem~\ref{thm:nfp}.
The $n$ eigenvalues of the perturbed pencil
\[
\widetilde A-\lambda\widetilde B:=A-\lambda B+\tau \, (UD_AU^*-\lambda \, UD_BU^*)
\]
can then be grouped into following three classes of eigenvalues:
\begin{enumerate}
\item \emph{True eigenvalues}: There are $r:=n-2M-k$ eigenvalues that are precisely the eigenvalues
of the original singular pencil $A-\lambda B$. If $\mu$ is such an eigenvalue with
corresponding right and left eigenvectors $x$ and $y$ then $U^*x=0$ and $U^*y=0$.
\item \emph{Prescribed eigenvalues}: There are $k$ semisimple eigenvalues that coincide with the $k$ semisimple eigenvalues
of $D_A-\lambda D_B$. If $\mu$ is such an eigenvalue with corresponding right and left eigenvectors $x$ and $y$ then
$U^*x\ne 0$ and $U^*y\ne 0$.
\item \emph{Random eigenvalues}: These are the remaining $2M$ eigenvalues. They are all nonreal and simple and if $\mu$ is
such an eigenvalue with the corresponding right and left $x$ and $y$,
then we either have $U^*x=0$ and $U^*y\ne 0$, or $U^*x\ne 0$ and $U^*y=0$. In that case, also $\overline{\mu}$ is an
eigenvalue and $x$ and $y$ are corresponding left and right eigenvectors, respectively.
\end{enumerate}
\end{summary}
\begin{remark}\rm
We highlight that the fact that all random eigenvalues of the perturbed Hermitian pencil
are nonreal is highly nontrivial. Indeed, it is shown in \cite{DTDD24} that the
generic eigenstructure of a regular Hermitian pencil may consist of real and nonreal
eigenvalues. Numerical experiments suggest that for random regular Hermitian pencils
the probability of having at least one real eigenvalue tends to $1$ when the size tends to
infinity (for odd size this is obvious, but it also seems to be the case for even size).
\end{remark}

\begin{remark}\label{rm:orthochoice}\rm
Instead of random matrices $U\in\mathbb C^{n,k}$, we will use matrices with
orthonormal columns in practice. At first glance one may think that this leads to a
different notion of genericity such that the results from Theorem~\ref{thm:nfp} and
Theorem~\ref{thm:mainBor} can no longer be applied, but in fact this is not an
issue. The key observation that generically a random matrix $U\in\mathbb C^{n,k}$
will be of full rank $k$. Then computing a thin QR decomposition $U=QR$, we find that
\[
U\,(D_A-\lambda D_B)\,U^*=QR\,(D_A-\lambda D_B)\,R^*Q^*,
\]
so the perturbation with $U$ and the prescribed pencil $D_A-\lambda D_B$ is exactly
the same as the perturbation with $Q$ and the prescribed pencil $R(D_A-\lambda D_B)R^*$ which is congruent to $D_A-\lambda D_B$ and hence has the same
prescribed eigenvalues. (This information on the spectrum was the only information on
$D_A-\lambda D_B$ that was used in the construction of the generic sets in the proofs
of Theorem~\ref{thm:nfp} and Theorem~\ref{thm:mainBor}.) Consequently, it is enough
to know that a matrix $U\in\mathbb C^{n,k}$ with orthonormal columns is the
isometric factor of a QR decomposition of a matrix in the generic sets from
Theorem~\ref{thm:nfp} and Theorem~\ref{thm:mainBor} to be able to apply both results.
Such a matrix can be obtained via the QR decomposition of a random
complex $n\times k$ matrix.
\end{remark}

\begin{remark}\label{rm:definite}\rm
If one of the matrices $A$ and $B$ is positive semidefinite and hence all eigenvalues
of the pencil are real and semisimple, then this property can be preserved by choosing
$D_A$ or $D_B$ positive definite, respectively, and $\tau>0$. (Similarly, the case
that a linear combination of $A$ and $B$ is positive semidefinite can be treated.)
Consequently,
the perturbed regular pencil $\widetilde A-\lambda\widetilde B$ will only have real
semisimple eigenvalues. Observe that this does not contradict Theorem~\ref{thm:mainBor}
or Summary~\ref{rm:3groups}, because in this case all minimal indices of the original
pencil $A-\lambda B$ are zero and hence the perturbed pencil only has true and
prescribed, but no random eigenvalues.
\end{remark}

\begin{remark}\label{rem:sign} \rm
It follows from~\eqref{eq:nfp} that the regular part of the original pencil
$A-\lambda B$ is contained in the perturbed pencil $\widetilde A-\lambda\widetilde B$.
As a consequence, the true eigenvalues of the perturbed pencil do not
only have the same multiplicities as the eigenvalues of the original pencil, but
also the sign characteristic is preserved. This observation can be used to calculate
the sign characteristic of true eigenvalues even in the case that the eigenvalues
and eigenvectors of the perturbed Hermitian pencil are computed by a method that does
not make use of the Hermitian structure of the problem. To this end, it only matters
that the perturbed regular Hermitian pencil was obtained using a structure-preserving
procedure.

For example, assume that $\lambda_0\in\mathbb R$ is a semisimple real eigenvalue
and that the eigenvectors $x_1,\dots,x_g$ form a basis for the corresponding
deflating subspace. Then setting $X := [\,x_1 \ \ \dots \ \ x_g\,]$, it remains to compute the inertia of the Hermitian
matrix $X^*\widetilde BX$, i.e., the numbers $\pi$ and $\nu$ of positive and negative
eigenvalues, respectively. It then follows immediately by inspection from the
Thompson canonical form, that the sign characteristic of the eigenvalue $\lambda_0$
contains $\pi$ times the sign $+1$ and $\nu$ times the sign $-1$. Similarly, the
sign characteristic of the eigenvalue $\infty$ can be calculated using the matrix
$\widetilde A$ instead of $\widetilde B$.

If $\lambda_0$ is not semisimple, then the calculation of the sign characteristic
is more complicated as principal vectors of higher order from Jordan chains are
needed, but again the information on the sign characteristic is preserved in the
transition from the singular Hermitian pencil $A-\lambda B$ to the perturbed
regular Hermitian pencil $\widetilde A-\lambda\widetilde B$.
\end{remark}

\section{Structure preserving projections of singular pencils}\label{sec:proj}
In this section, we will show that instead of solving the singular generalized
Hermitian eigenvalue problem by applying the rank-updating
perturbation~\eqref{pert1awithtau},
we can alternatively compute the eigenvalues and eigenvectors of a projected pencil
$(n-k)\times (n-k)$ Hermitian pencil $W^*(A-\lambda B)\,W$, where
$W\in\mathbb C^{n,n-k}$ is a random complex
matrix. (Using an argument similar to the one in Remark~\ref{rm:orthochoice},
the matrix $W$ can be chosen to have orthonormal columns, i.e., $W^*W=I_{n-k}$.)
This is a generalization
of the theory from \cite{HMP_SingGep2} to the structured problem with Hermitian matrices.
Before we state the main theorem of the section, we repeat the following definition and result from \cite{HMP_SingGep2}.

\begin{definition}[{\cite[Def.~9.1]{HMP_SingGep2}}]
Let $\Omega\subseteq\C^{n,k}$. Then the \emph{pointwise orthogonal complement}
$\Omega^\perp$ of $\Omega$ in $\C^{n,n}$ is the set
$\Omega^\perp:=\{W\in \C^{n,n-k}\ |\ \exists \, U\in\Omega, \, W^*U=0\}$.
\end{definition}
\begin{proposition}[{\cite[Prop.~9.3]{HMP_SingGep2}}]\label{prop:perpgen}
Let $\Omega\subseteq\C^{n,k}$ be a generic set. Then also $\Omega^\perp\subseteq\C^{n,n-k}$
is a generic set.
\end{proposition}

The following result is a structure-preserving extension to \cite[Thm.~4.2]{HMP_SingGep2}.
\begin{theorem}\label{main_projection}
Let $A-\lambda B$ be a singular $n\times n$ Hermitian matrix pencil having normal rank $n-k$.
Then there exists a generic set $\Omega\subseteq\C^{n,n-k}$ with the property that for each $W\in\Omega$ the following statements hold:
\begin{enumerate}
\item The $(n-k)\times(n-k)$ pencil $W^*(A-\lambda B)\,W$ is regular and such that
 its Kronecker canonical form is
\[
\left[\begin{array}{cc}R(\lambda)&0\\ 0&R_{\rm ran}(\lambda)\end{array}\right],
\]
where $R(\lambda)$ coincides with the regular part of $A-\lambda B$ and where all eigenvalues of $R_{\rm ran}(\lambda)$ are simple
and distinct from the eigenvalues of $R(\lambda)$.
\item Let the columns of $W_\perp\in \C^{n,k}$ form a
basis of the orthogonal complement of the range of $W$ and let
$\lambda_0$ be an eigenvalue of $W^*(A-\lambda B)\,W$ with corresponding left eigenvector $y$ and right eigenvector $x$.
If $\lambda_0\in \C$, then $\lambda_0$ is an eigenvalue of $A-\lambda B$, i.e., of $R(\lambda)$, if and only if
$y^*W^*(A-\lambda_0 B)\,W_\perp=0$ and $W_\perp^*(A-\lambda_0 B)\,Wx= 0$. If $\lambda_0=\infty$ then $\lambda_0$ is an eigenvalue of
$A-\lambda B$, i.e., of $R(\lambda)$, if and only if
$y^*W^*BW_\perp=0$ and $W_\perp^*BWx= 0$.
\end{enumerate}
\end{theorem}

\begin{proof}
We start by showing that there exists a generic set $\Omega_r\subseteq\C^{n,n-k}$ such that $W^*(A-\lambda B)\,W$ is a regular pencil for each $W\in\Omega_r$.
To see this, observe that the determinant of $W^*(A-\lambda B)\,W$ is a polynomial
in the real and imaginary parts of the entries of $W$.
Let us show that this polynomial is not identically zero.
Since $A-\lambda B$ has normal rank $n-k$, there exists $\lambda_0\in\mathbb R$
such that $\rank(A-\lambda_0 B)=n-k$. Since
$A-\lambda_0 B$ is Hermitian, there exist a unitary matrix $Q=[Q_1\ Q_2]$,
where $Q_1\in\mathbb C^{n,n-k}$, and a diagonal matrix
$D=\diag(D_1,0)$, where $D_1$ is a $(n-k)\times (n-k)$ nonsingular diagonal
matrix, such that $Q^*(A-\lambda_0B)Q=D$. If we take $W=Q_1$, then
$\det(W^*(A-\lambda_0 B)\,W)\ne 0$ and $W^*(A-\lambda B)\,W$ is a regular Hermitian pencil. As a consequence, the set of all $W\in\mathbb C^{n,n-k}$ for which
$W^*(A-\lambda B)W$ is regular is a generic set.

Next, let $D_A,D_B\in\mathbb C^{k,k}$ be Hermitian such that the pencil
$D_A-\lambda D_B$ is regular with semisimple eigenvalues that are distinct from the
eigenvalues of $A-\lambda B$. Set
$\Omega:=(\Omega_1\cap\Omega_2\cap\Omega_4)^\perp\cap\Omega_r\subseteq\C^{n,n-k}$, where $\Omega_1$, $\Omega_2$ and $\Omega_4$ are the
generic subsets of $\C^{n,k}$ from Proposition~\ref{prop:genpert},
Theorem~\ref{thm:nfp} and Theorem~\ref{thm:mainBor}, respectively,
applied to $A-\lambda B$. Then $\Omega$ is generic. For a fixed $W\in\Omega$ there
exists a matrix $U\in\Omega_1\cap\Omega_2\cap\Omega_4$ such that $W^*U=0$.
Since $U\in\Omega_1$, if follows from Proposition \ref{prop:genpert} than there exists a nonsingular matrix $P\in \C^{n,n}$ such that
\begin{equation}\label{eq:7.8.19}
P\,(A-\lambda B)\,P^*=\left[\begin{array}{cc}R(\lambda)&0\\ 0& S(\lambda)\end{array}\right]\quad\mbox{and}\quad
PU=\left[\begin{array}{c}0\\ \widetilde U\end{array}\right],
\end{equation}
where $R(\lambda)$ and $S(\lambda)$ are the regular and singular parts of $A-\lambda B$, respectively, and
$PU$ is partitioned conformably with $P\,(A-\lambda B)\,P^*$.

Let $R(\lambda)$ have the size $r\times r$, i.e.,
$\widetilde U\in \C^{n-r,k}$. Since $U$ has full column rank $k$, there exists
a nonsingular matrix $T\in \C^{n-r,n-r}$ such that with $\widetilde P:=\operatorname{diag}(I_r,T)\cdot P$ we get
\begin{equation}\label{eq:proof_project_PABP}
\widetilde P\,(A-\lambda B)\,\widetilde P^*=\left[\begin{array}{ccc}R(\lambda)&0&0\\ 0&S_{11}(\lambda)&S_{12}(\lambda)\\ 0&S_{21}(\lambda)&S_{22}(\lambda)
\end{array}\right]\quad\mbox{and}\quad \widetilde PU=\left[\begin{array}{c}0\\ 0\\ I_k\end{array}\right],
\end{equation}
where $S_{11}(\lambda)\in \C^{n-k-r,n-k-r}$ and $\widetilde PU$ is partitioned conformably with $\widetilde P\,(A-\lambda B)\,\widetilde P^*$.
If we partition $W^*\widetilde P^{-1}=[W_1^* \ \, W_2^* \ \, W_3^*]$,
where $W_1\in \C^{r,n-k}$, $W_2\in \C^{n-k-r,n-k}$, and $W_3\in \C^{k,n-k}$,
then if follows from $W^*U=(W^*\widetilde P^{-1})(\widetilde PU)=0$ that $W_3=0$. Therefore, since $W^*$ has full rank,
the matrix $\widetilde W^*:=[W_1^* \ \, W_2^*]$ is invertible.

From $W^*(A-\lambda B)\,W= W^*\widetilde P^{-1}\widetilde P\,(A-\lambda B)\,\widetilde P^*
\widetilde P^{-*}W$ and \eqref{eq:proof_project_PABP}
we now obtain
\[
W^*(A-\lambda B)\,W
=\widetilde W^*
\left[\begin{array}{cc}R(\lambda)&0\\ 0&S_{11}(\lambda)\end{array}\right]
\widetilde W.
\]
Since
$\widetilde W$ is invertible,
it follows that the pencil $W^*(A-\lambda B)\,W$ is equivalent to the pencil $\operatorname{diag}\big(R(\lambda), S_{11}(\lambda)\big)$.
Furthermore, since $U\in\Omega_2\cap\Omega_4$ it follows that the pencil
\[
\widetilde A-\lambda \widetilde B:= A-\lambda B+\tau \, (UD_AU^*-\lambda \, UD_BU^*),
\]
where $\tau\ne 0$, is regular and has the Kronecker canonical form
\[
\left[\begin{array}{ccc}R(\lambda)&0&0\\ 0&R_{\rm pre}(\lambda)&0\\ 0&0&R_{\rm ran}(\lambda)\end{array}\right],
\]
where $R_{\rm pre}(\lambda)$ and $R_{\rm ran}(\lambda)$ are the
regular parts containing the prescribed and random eigenvalues according to
Theorem~\ref{thm:mainBor} and Remark~\ref{rm:3groups}, respectively.
We can now show that $S_{11}(\lambda)$ and $R_{\rm ran}(\lambda)$ are equivalent.
Since the blocks $S_{11}(\lambda)$ and $R_{\rm ran}(\lambda)$ are of the same size
and all eigenvalues of $R_{\rm ran}(\lambda)$ are simple, it is enough to show that
each eigenvalue of $R_{\rm ran}(\lambda)$ is also an eigenvalue of
$S_{11}(\lambda)$.

Let $\lambda_0$ be an eigenvalue of $R_{\rm ran}(\lambda)$ (generically, this
eigenvalue will be finite, i.e., we have $\lambda_0\in\mathbb C$)
and let $z$ be a right eigenvector and $v$ be a left eigenvector of $\widetilde A-\lambda \widetilde B$
associated with $\lambda_0$. By Theorem~\ref{thm:mainBor} then either $U^*z=0$ or $U^*v=0$
(exactly one of these statements is true) and thus, since the columns of $W$ are
a basis of the orthogonal
complement of the range of $U$, then either
$z=W x$ for a nonzero $x\in \C^{n-k}$ or $v=Wy$ for a nonzero $y\in \C^{n-k}$ (again, exactly one of
these statements is true). Then either $y^*W^*(A-\lambda_0 B)\,W=0$ or $W^*(A-\lambda_0 B)\,Wx=0$ and thus $\lambda_0$ is an eigenvalue of $W^*(A-\lambda B)\,W$. This finishes the proof of $1)$.

For the proof of $2)$, first observe that if $\lambda_0$ is an eigenvalue of $R(\lambda)$ and $z$ and $v$ are right
and left eigenvectors of $\widetilde A-\lambda \widetilde B$
associated with $\lambda_0$, respectively, then
following the same argument as in the paragraph above, it follows from Theorem
\ref{thm:nfp} that both $z=Wx$ and
$v=Wy$ for nonzero vectors $x,y\in \C^{n-k}$.
The map $z\mapsto x$ is a bijection from the set of right eigenvectors
 of $\widetilde A-\lambda\widetilde B$ associated with $\lambda_0$ to the set of right eigenvectors of $W^*(A-\lambda B)\,W$ associated
 with $\lambda_0$. Similar observations hold for the map $w\mapsto y$, related
 to the set of left eigenvectors.

Thus, if $\lambda_0\in \C$ is an eigenvalue of $W^*(A-\lambda B)\,W$ such
that $\lambda_0$ is an eigenvalue of $R(\lambda)$ with left eigenvector $y$ and right eigenvector $x$,
then both $Wy$ and $Wx$ are left respectively right eigenvectors of $\widetilde A-\lambda \widetilde B$ associated with $\lambda_0$. Then
$(A-\lambda_0 B)\,Wx=0$ and this implies
$W_\perp^*(A-\lambda_0 B)\,W x=0$. In a similar way
$y^*W^*(A-\lambda_0 B)=0$ implies that $y^*W^*(A-\lambda_0 B)\,W_\perp=0$.

If on the other hand $\lambda_0$ is an eigenvalue of
$R_{\rm ran}(\lambda)$, then either $Wy$ or $Wx$ is a left or right eigenvector of
$A-\lambda B$, respectively,
but not both. First assume, that $Wy$ is a left eigenvector of $\widetilde A-\lambda \widetilde B$ associated with $\lambda_0$. Then
$W_\perp^*(A-\lambda_0 B)\,Wx\ne 0$, because otherwise, keeping in mind
that $y$ is a left eigenvector of $W^*(A-\lambda B)W$ and $[W \ \, W_\perp]$ is nonsingular, we would have
$(A-\lambda_0 B)\,Wx=0$ implying that
$Wx$ is a right eigenvector of $A-\lambda B$ which is a contradiction.
Analogously, we show that $y^*W^*(A-\lambda_0 B)\,W_\perp \ne 0$ when $Wx$ is a right eigenvector of $A-\lambda B$
associated with $\lambda_0$. The claim in the case that $\lambda_0=\infty$ is an eigenvalue of $W^*(A-\lambda B)\,W$
follows in a similar way.
\end{proof}
\begin{remark}\rm
We highlight that if all minimal indices of the singular pencil are zero, then
the regular part has size $n-k$ and hence any projected pencil as in
Theorem~\ref{main_projection} will have no random eigenvalues which makes the identification of the true eigenvalues trivial. For example this is the case
if any linear combination of $A$ and $B$ is positive semidefinite, cf. Remark~\ref{rm:definite}.
\end{remark}
\begin{remark} \label{rem:augm} \rm
Without giving details, we mention that the augmentation approach from Part II \cite{HMP_SingGep2} can also be adapted for a structure-preserving approach.
The augmented pencil in the symmetric case is then of the form  (cf.~\cite[(5.1)]{HMP_SingGep2})
\[
\mtxa{cc}{A & UD_A \\ D_A^*U^* & 0} \mtxa{c}{x \\ z} = \lambda \, \mtxa{cc}{B & UD_B \\ D_B^*U^* & 0} \mtxa{c}{x \\ z}.
\]
Here we explicitly select the matrices $D_A$ and $D_B$ to be diagonal such that the
regular pencil $D_A - \lambda D_B$ has simple complex eigenvalues without complex
conjugate pairs, to avoid getting double eigenvalues.
(This is a noticeable difference with Part II, where the diagonal entries of these matrices are typically chosen from the interval $[1,\,2]$.)
This augmentation approach may especially be useful for large sparse problems with modest values of $k$; we leave this for future work.
\end{remark}

\section{Other symmetry structures}\label{sec:symm}
Although the main results of the previous two section were formulated for Hermitian
pencils only, they can be directly applied to many other symmetry structures that
are related to the Hermitian structure. We first recall the following definition,
see, e.g., \cite{MacMMM06b}.

\begin{definition}\label{def:struc}
Let $A,B\in\mathbb C^{n,n}$. Then the pencil $A-\lambda B$ is called
\begin{itemize}
    \item \emph{$*$-even} if $A=A^*$ and $B=-B^*$,
    \item \emph{$*$-odd} if $A=-A^*$ and $B=B^*$,
    \item \emph{skew-Hermitian} if $A=-A^*$ and $B=-B^*$,
    \item \emph{$*$-palindromic} if $B=A^*$, and
    \item \emph{$*$-anti-palindromic} if $B=-A^*$.
\end{itemize}
\end{definition}

Since all these symmetry structures can be easily reduced to the Hermitian
structure, there are straightforward generalizations of the main results from the
previous two sections for these symmetry structures. We refrain from stating them
explicitly for each individual symmetry structure, but formulate the following
meta-theorem instead.

\begin{mtheorem}
Let $A-\lambda B$ be a singular $n\times n$ complex matrix pencil of normal
rank $n-k$ carrying one of
the symmetry structures as in Definition~\ref{def:struc}. Then the results
from Theorem~\ref{thm:nfp}, Theorem~\ref{thm:mainBor} and
Theorem~\ref{main_projection} hold for $A-\lambda B$, where for the first
two theorems the regular $k\times k$ pencil $D_A-\lambda D_B$ is chosen to
carry the same symmetry structure as the original pencil $A-\lambda B$.
\end{mtheorem}

\begin{proof}
If $A-\lambda B$ is even, odd, or skew-Hermitian, the result immediately
follows from the observation that multiplying one (or both) of the coefficient
matrices by the imaginary unit $i$ turns the pencil into a Hermitian pencil.

In the case of an (anti-)palindromic pencil $A-\lambda A^*$ we can apply
a Cayley transformation; see, e.g., \cite{LanR95}. These transformations map
(anti-)palindromic pencils to even or odd pencils and vice versa; see, e.g.,
\cite{MacMMM06b}. Since Cayley transformations are special M\"obius transformations, it follows from the results in \cite{MMMM15} that the
eigenvalues change in a predictable bijective way while the eigenvectors
remain invariant. Also the change in the sign characteristic under a M\"obius
transformation is well understood \cite{MehNTX16}. Hence the result
follows by applying the Cayley transformation and a multiplication of one
of the coefficient matrices by $i$ to both the original pencil and the
perturbation pencil to turn them into Hermitian pencils.
\end{proof}

The situation changes drastically if the considered pencils are real symmetric
and real symmetric rank-completing perturbations or real projections
are considered. Indeed, inspecting the proof of Lemma~\ref{lem:Lo}, it was explicitly used that the
considered perturbation matrices are allowed to be complex to
show that the values $\alpha_1,\dots,\alpha_M$ are different from each of their complex conjugates. Numerical examples show that this is no longer the case if real
symmetric pencils and real
symmetric rank-completing perturbations are considered. In fact, experiments
suggest that $\{\alpha_1,\dots,\alpha_m\}=\{\overline{\alpha}_1,\dots,\overline{\alpha}_M\}$.
These values turn out to be double random eigenvalues of the
perturbed Hermitian pencil. As a small illustrating example, consider
\[
A - \lambda B = \smtxa{rrr}{0 & 1 & -\lambda \\ 1 & 0 & 0 \\ -\lambda & 0 & 0},
\]
a pencil with right (and therefore also left) minimal index equal to 1.
The following special perturbations describe the behavior that can also be observed
for generic rank-completing perturbations:
for the real perturbation  $u = [1, 1, 1]^T$, the eigenvalues of the pair $(A + 2uu^*, \, B+uu^*)$ are the prescribed eigenvalue $2$, and the double (``random'') eigenvalue $1$.
The geometric multiplicity of this double eigenvalue is 1, and the corresponding eigenvector (since the eigenvalue is real, the left and right eigenvectors are equal) satisfies the orthogonality condition $u^*x = 0$.
With a Hermitian perturbation using $u = [1, 1, 1]^T + i \, [1,2,3]^T$, we get the prescribed eigenvalue 2, and two complex conjugate (``random'') eigenvalues $1.4 \pm 0.2\,i$; these eigenvalues are therefore simple.
For one of the two eigenvalues we have $u^*x = 0$ and $u^*y \ne 0$ (where $y$ is the corresponding left eigenvector); for the other we have $u^*x \ne 0$ and $u^*y = 0$.

As the next example, we consider a pencil with right minimal index 2:
\[
A - \lambda B = \smtxa{rrrrr}{0 & 0 & 1 & -\lambda & 0 \\ 0 & 0 & 0 & 1 & -\lambda \\ 1 & 0 & 0 & 0 & 0 \\ -\lambda & 1 & 0 & 0 & 0 \\ 0 & -\lambda & 0 & 0 & 0}.
\]
A real symmetric rank-1 perturbation using $u = [1, 1, 1, 1, 1]^T$ yields the prescribed eigenvalue 2, as well as a double eigenvalue $\frac12 + \frac12 \sqrt3\,i$, and a double eigenvalue $\frac12 - \frac12 \sqrt3\,i$.
The geometric multiplicity of the two double eigenvalues is 1, and the corresponding left and right eigenvector satisfy the orthogonality condition $u^*x = 0$ and $u^*y=0$, note that for a real symmetric pencil a left eigenvector for a complex eigenvalue $\lambda$ is a right eigenvector
for $\overline \lambda$.
A Hermitian rank-1 perturbation with $u = [1, 1, 1, 1, 1]^T + i \, [1, 2, 3, 4, 5]^T$ gives the eigenvalue $2$ and four simple eigenvalues in two complex conjugate pairs ($0.59 \pm 1.14\,i$ and $0.71 \pm 1.04\,i$).
Of every of these four eigenvalues, exactly one of the quantities $u^*x$ and $u^*y$ is zero.

For various random real rank-1 perturbations, we observe either two double
real random eigenvalues, or a pair of complex conjugate double random eigenvalues.
For Hermitian rank-1 perturbations, the random eigenvalues are all simple, non-real, and come in complex conjugate pairs.
This gives rise to the following conjecture that has already been proven for the case of 1 minimal index and a rank-1 perturbation in \cite{mehl2017parameter}.

\begin{conjecture} \label{conj:double}
Let $A-\lambda B$ be a real symmetric singular $n\times n$ pencil with normal rank
$n-k$ and let $D_A-\lambda D_B$ be a real symmetric regular $k\times k$ pencil. Then
there exists a generic set $\Omega\subseteq\mathbb R^{n,k}$ such that for all
$U\in\Omega$ and all $\tau\in\mathbb R\setminus\{0\}$ the spectrum of the pencil
$A-\lambda B+\tau \, U\,(D_A-\lambda D_B)\,U^\top$ consists of the eigenvalues of
$A-\lambda B$ (true eigenvalues), the eigenvalues of $D_A-\lambda D_B$ (prescribed
eigenvalues) and random eigenvalues that all have algebraic multiplicity precisely
two.
\end{conjecture}

Further investigation of this conjecture is out of scope for this paper,
and is left for future research.

\section{Numerical experiments} \label{sec:appl}
We will now show two experiments and a new symmetric determinantal representation for bivariate polynomials.
\begin{experiment} \rm
For the first experiment, we consider the singular real symmetric $24\times 24$
pencil
\begin{equation}\label{22.5.24}
A-\lambda E = S\cdot \smtxa{ccccc}{R(\lambda)&0&0&0&0\\
0&0&L_1(\lambda)&0&0\\ 0&L_1(\lambda)^\top &0&0&0\\ 0&0&0&0&L_2(\lambda)\\
0&0&0&L_2(\lambda)^\top&0}\cdot S^\top,
\end{equation}
where $S\in\mathbb R^{24,24}$ is random and where the regular part
$R(\lambda)=R_A-\lambda R_E$ is given by
\begin{align*}
R_A & =\text{diag}\Big(1,2,2,3,
\smtxa{rr}{1&0\\0&\!-1\!},
\smtxa{rr}{1&1\\1&\!-1\!},
\smtxa{rr}{1&2\\2&\!-1\!},
\smtxa{rr}{2&0\\0&\!-2\!},
\smtxa{rr}{2&1\\1&\!-2\!},
\smtxa{rr}{2&2\\2&\!-2\!}\Big), \\[1mm]
R_E & =\text{diag}\Big(1,1,-1,-1,
\smtxa{rr}{0&1\\1&0},
\smtxa{rr}{0&1\\1&0},
\smtxa{rr}{0&1\\1&0},
\smtxa{rr}{0&1\\1&0},
\smtxa{rr}{0&1\\1&0},
\smtxa{rr}{0&1\\1&0}\Big).
\end{align*}
Thus, the pencil is singular with normal rank $22$ and has the simple complex eigenvalues
$\pm i$, $\pm 2i$, $1\pm i$, $1\pm 2i$, $2\pm i$, $2\pm 2i$, as well as the real
eigenvalues $1,2,3$. The simple real eigenvalues $1$ and $3$ have the signs $+1$ and $-1$,
respectively, and the double semisimple eigenvalue $2$ has the sign characteristic $(1,-1)$. Applying a Hermitian rank-completing perturbation with a complex matrix
$U\in\mathbb C^{24,2}$ having orthonormal columns and a $2 \times 2$ real diagonal
pencil $D_A-\lambda D_E$ yields the results in Table~\ref{tab:2}, where we have used
the standard implementation of the $QZ$ algorithm by calling the
Matlab routine {\sf eig}.
Note that although it may not be suitable for all cases, for our reported experiments we may take $D_A$ and $D_B$ diagonal (and hence real).

\begin{table}[htb!] \label{tab:2}
\centering
\caption{Results of a Hermitian rank-completing perturbation applied to the pencil
$A-\lambda E$ from\eqref{22.5.24} }
{\footnotesize \begin{tabular}{r|rlll} \hline \rule{0pt}{2.3ex}%
$j$ & $\lambda_j~~~~~~~~~$ & $\ \ \|U^*x_j\|$ & $\ \ \|U^*y_j\|$ & Type \\[0.5mm]
\hline \rule{0pt}{2.6ex}%
  1 &   3.0000+0.0000i & {$2.58\cdot 10^{-14}$} &   {$3.68\cdot 10^{-14}$} & True\\
  2 &   2.0000+0.0000i &{$4.26\cdot 10^{-14}$}  &{$1.15\cdot 10^{-13}$} & True\\
  3 &   1.0000+1.0000i &{$9.48\cdot 10^{-15}$}  &  {$9.33\cdot 10^{-15}$} & True\\
  4 &   1.0000$-$1.0000i &{$1.80\cdot 10^{-14}$}  & {$8.68\cdot 10^{-15}$} & True\\
  5 &   1.0000+2.0000i &{$1.11\cdot 10^{-14}$}  & {$5.67\cdot 10^{-15}$} & True\\
  6 &   1.0000$-$2.0000i &{$4.20\cdot 10^{-15}$}  & {$8.64\cdot 10^{-15}$} & True\\
  7 &   2.0000+1.0000i &{$1.71\cdot 10^{-14}$}  & {$1.61\cdot 10^{-14}$} & True\\
  8 &   2.0000$-$1.0000i &{$2.09\cdot 10^{-14}$}  &  {$2.27\cdot 10^{-14}$} & True\\
  9 &   2.0000+0.0000i &{$1.18\cdot 10^{-13}$}  &  {$3.02\cdot 10^{-14}$} & True\\
 10 &   2.0000$-$2.0000i &{$1.81\cdot 10^{-14}$}  & {$9.99\cdot 10^{-15}$} & True\\
 11 &   2.0000+2.0000i &{$5.96\cdot 10^{-15}$} & {$1.65\cdot 10^{-14}$} & True\\
 12 &   1.0000$-$0.0000i &{$1.45\cdot 10^{-14}$}  &{$1.14\cdot 10^{-14}$} & True\\
 13 &   0.0000$-$2.0000i &{$1.37\cdot 10^{-14}$}  &{$4.39\cdot 10^{-15}$} & True\\
 14 &   0.0000+2.0000i &{$6.26\cdot 10^{-15}$}  &{$9.86\cdot 10^{-15}$} & True\\
 15 &   0.0000+1.0000i &{$3.40\cdot 10^{-15}$}  &{$1.30\cdot 10^{-14}$} & True\\
 16 &   0.0000$-$1.0000i &{$1.37\cdot 10^{-14}$}  & {$2.24\cdot 10^{-15}$} & True\\
 17 &  $-$0.5005$-$0.4835i &{$3.77\cdot 10^{-15}$}  &   {$1.03\cdot 10^{-1}$} & Random\\
 18 &  $-$0.3654$-$0.3209i &{$2.73\cdot 10^{-15}$}  &   {$5.47\cdot 10^{-2}$} & Random\\
 19 &  $-$0.2505+0.8105i &{$ 4.65\cdot 10^{-15}$}  &   {$1.62\cdot 10^{-2}$} &  Random\\
 20 &  $-$0.5005+0.4835i &{$1.03\cdot 10^{-1}$}  & {$3.55\cdot 10^{-15}$} & Random\\
 21 &  $-$0.3654+3.2093i &{$5.47\cdot 10^{-2}$}  &{$2.18\cdot 10^{-15}$} &  Random\\
 22 &  $-$0.2505$-$0.8105i &{$1.62\cdot 10^{-2}$}  &{$3.96\cdot 10^{-15}$} &  Random\\
23 &    1.8100+0.0000i & $1.59\cdot 10^{-1}$ & $1.59\cdot 10^{-1}$ & Prescribed\\
 24 &   2.2343+0.0000i & $1.42\cdot 10^{-1}$ & $1.42\cdot 10^{-1}$ & Prescribed\\
\hline
\end{tabular}}
\end{table}

As desired, the algorithm classifies all eigenvalues of the pencil $A-\lambda E$ as true eigenvalues and in addition detects two prescribed eigenvalues as well as six random eigenvalues
that turn out to be all complex and simple, as expected. In particular, this example shows
that real eigenvalues among the true and prescribed eigenvalues do not cause any problems in the numerical method.
For the real eigenvalues $\lambda_j$ with $j=1,2,9,12$ we observe that the values
$\|U^*x_j\|$ and $\|U^*y_j\|$ are slightly different. This is due to the fact that the
routine {\sf eig} ignores the Hermitian structure of the problem and computes the right and
left eigenvectors $x_j$ and $y_j$ independently. Therefore, we only have
$y_j\approx e^{i\phi}x_j$ for some $\phi\in [0,2\pi)$ which leads to different
rounding errors in the values $U^*x_j$ and $U^*y_j$.

Although the $QZ$ algorithm did not make use of the Hermitian structure of the
problem, we can calculate the sign characteristic of the singular Hermitian pencil
$A-\lambda E$ from the perturbed regular Hermitian pencil as discussed in
Remark~\ref{rem:sign}. Considering again the real eigenvalues $\lambda_1=3$,
$\lambda_2=\lambda_9=2$, $\lambda_{12}=1$ and the corresponding right eigenvectors,
we obtain $x_1^*Ex_1=-1.1865$ and $x_{12}^*Ex_{12}=0.1456$ while the $2\times 2$ matrix
\[
[\,x_2 \ \ x_9\,]^*\, E\
[\,x_2 \ \ x_9\,]
\]
has the eigenvalues $-1.5649$ and $0.2395$. Thus, we recover the sign $-1$ for the
eigenvalue $\lambda_1=3$, the sign $+1$ for the eigenvalue $\lambda_{12}=1$ and the
sign characteristic $(1,-1)$ for the double eigenvalue $\lambda_2=\lambda_9=2$.

If, in contrast, a real symmetric rank-completing perturbation is used, then the
algorithm fails in recognizing the random eigenvalues. Instead, it classifies
the eigenvalues wrongly into $18$ true ones and $6$ prescribed ones.
\end{experiment}

Next, we show that the developed techniques can be exploited for the solution of bivariate polynomial systems.
In \cite{PH_BiRoots_SISC,BDDHP_Uniform}, bivariate polynomial systems are solved via determinantal representations.
Nonsymmetric representations have been used in Part I \cite[Ex.~7.1]{HMP19} to solve such a system.
In this section we introduce new {\em symmetric} determinantal representations, for which we will use the perturbation method introduced in this paper.
Given a bivariate polynomial $p(\lambda,\mu)$, we are interested in matrices $A, B, C$ such that $p(\lambda,\mu) = \det(A+\lambda B+\mu C)$.
For polynomials of degree $d$, the nonsymmetric representations of \cite{BDDHP_Uniform} are of size $2d-1$, the optimal known ``uniform'' size (see \cite{BDDHP_Uniform} for more information).
The symmetric representations that we will introduce now have the disadvantage of being {\em asymptotically} of size ${\mathcal O}(d^2)$ (similar to those of \cite{PH_BiRoots_SISC}). However, for small $d$, say $d \le 5$, they are of the same or only slightly larger size compared to nonsymmetric representations.

More specifically, for the cubic bivariate polynomial $a_{00}+a_{10}\lambda+a_{01}\mu+a_{20}\lambda^2+a_{11}\lambda\mu+a_{02}\mu^2 + a_{30}\lambda^3+a_{21}\lambda^2\mu+a_{12}\lambda\mu^2+a_{03}\mu^3$, we get the following $5 \times 5$ determinantal representation:
\[
\mtxa{ccccc}{
a_{00}+a_{10}\,\lambda+a_{01}\,\mu & \frac12 a_{11}\,\mu & -\lambda & & -\mu \\
\frac12 a_{11}\,\mu & a_{20}+a_{30}\,\lambda+a_{21}\,\mu & \ph{-}1 & & \\
-\lambda & \ph{-}1 & & &\\
& & & a_{02}+a_{12}\,\lambda+a_{03}\,\mu & \ph{-}1 \\
-\mu & & & \ph{-}1 &}.
\]
Therefore, for $d=3$, this is the optimal size $2d-1$ as in \cite{BDDHP_Uniform}.
We note that this representation is far from unique;
for instance, we can replace the $a_{20}$ element on position (2,2) by $\frac12 a_{20}\,\lambda$ elements on positions (1,2) and (2,1).
Additionally, we can carry out a symmetric permutation of rows and columns.
For degree $d=5$, for instance, one may check that we get an $11 \times 11$ symmetric representation.

\begin{experiment} \rm
Consider the cubic bivariate polynomial system
\begin{align*}
1+2\lambda+3\mu+4\lambda^2+5\lambda\mu+6\mu^2+7\lambda^3+8\lambda^2\mu+9\lambda\mu^2+10\mu^3 & = 0, \\
10+9\lambda+8\mu+7\lambda^2+6\lambda\mu+5\mu^2+4\lambda^3+3\lambda^2\mu+2\lambda\mu^2+\phantom{0}\mu^3 & = 0.
\end{align*}
We can construct {\em symmetric} determinantal representations of size $5 \times 5$ as above, which are of the same size as the nonsymmetric ones in \cite{BDDHP_Uniform}.
These are of the form $A_i + \lambda B_i + \mu C_i$ for $i = 1, 2$.
The $\lambda$-components of the roots $(\lambda, \mu)$ are the eigenvalues of the pencil $(\Delta_1, \, \Delta_0) = (C_1 \otimes A_2 - A_1 \otimes C_2, \, B_1 \otimes C_2 - C_1 \otimes B_2)$, and are presented in Table~\ref{tab:1}.
We have $\rank(\Delta_1) = 23$, $\rank(\Delta_0) = 18$, and $\text{nrank}(\Delta_1, \Delta_0) = 23$.
Since no linear combination of $\Delta_1$ and $\Delta_0$ is positive semidefinite, the eigenvalues are not all real.

\begin{table}[htb!] \label{tab:1}
\centering
\caption{Results of a symmetric rank-completing perturbation applied to the pencil $(\Delta_1, \, \Delta_0)$.}
{\footnotesize \begin{tabular}{c|cllll} \hline \rule{0pt}{2.3ex}%
$j$ & $\lambda_j$ & $\ \ |y_j^*x_j|$ & $\ \ \|U^*x_j\|$ & $\ \ \|U^*y_j\|$ & Type \\[0.5mm]
\hline \rule{0pt}{2.5ex}%
1 & $-2.4 -4.8\cdot 10^{-16}i$ & $1.1\cdot 10^{-2}$ & $1.8\cdot 10^{-16}$ & $3.3\cdot 10^{-16}$ & True finite \\
2--3 & $-0.56 \pm 2i$ & $4.2\cdot 10^{-5}$ & $3.5\cdot 10^{-16}$ & $4.4\cdot 10^{-16}$ & True finite \\
4--5 & $-1.1 \pm 0.3i$ & $3.2\cdot 10^{-3}$ & $3.5\cdot 10^{-16}$ & $3.9\cdot 10^{-16}$ & True finite \\
6--7 & $0.081 \pm 1.1i$ & $1.5\cdot 10^{-4}$ & $2.7\cdot 10^{-16}$ & $7.8\cdot 10^{-16}$ & True finite \\
8--9 & $0.072 \pm 1.2i$ & $4.1\cdot 10^{-3}$ & $6.0\cdot 10^{-16}$ & $5.2\cdot 10^{-16}$ & True finite \\
10 & $-3.7\cdot 10^6 +8.3\cdot 10^6i$ & $4.0\cdot 10^{-18}$ & $4.5\cdot 10^{-16}$ & $1.1\cdot 10^{-16}$ & True infinite \\
11 & $\phantom-3.7\cdot 10^6 -8.3\cdot 10^6i$ & $3.8\cdot 10^{-18}$ & $4.1\cdot 10^{-16}$ & $1.2\cdot 10^{-16}$ & True infinite \\
12 & $-3.9\cdot 10^3 +5.8\cdot 10^2i$ & $4.6\cdot 10^{-18}$ & $4.4\cdot 10^{-16}$ & $5.3\cdot 10^{-17}$ & True infinite \\
13 & $\phantom-3.9\cdot 10^3 -5.8\cdot 10^2i$ & $5.9\cdot 10^{-18}$ & $3.7\cdot 10^{-16}$ & $8.7\cdot 10^{-17}$ & True infinite \\
14 & $-5.8\cdot 10^2 -3.9\cdot 10^3i$ & $4.0\cdot 10^{-18}$ & $4.1\cdot 10^{-16}$ & $7.7\cdot 10^{-17}$ & True infinite \\
15 & $\phantom-5.8\cdot 10^2 +3.9\cdot 10^3i$ & $3.7\cdot 10^{-18}$ & $4.1\cdot 10^{-16}$ & $4.3\cdot 10^{-17}$ & True infinite \\
16 & $\infty$ & $1.9\cdot 10^{-18}$ & $4.4\cdot 10^{-16}$ & $2.5\cdot 10^{-16}$ & True infinite \\
17 & $\infty$ & $1.1\cdot 10^{-18}$ & $4.9\cdot 10^{-16}$ & $1.7\cdot 10^{-16}$ & True infinite \\
18 & $\infty$ & $1.6\cdot 10^{-18}$ & $5.1\cdot 10^{-16}$ & $1.8\cdot 10^{-16}$ & True infinite \\
19 & $\infty$ & $1.5\cdot 10^{-18}$ & $4.2\cdot 10^{-16}$ & $2.0\cdot 10^{-16}$ & True infinite \\
20 & $\infty$ & $3.6\cdot 10^{-19}$ & $3.6\cdot 10^{-16}$ & $1.9\cdot 10^{-16}$ & True infinite \\
21 & $\infty$ & $9.6\cdot 10^{-19}$ & $4.0\cdot 10^{-16}$ & $1.8\cdot 10^{-16}$ & True infinite \\
22 & $1.3+4.5\cdot 10^{-16}i$ & $1.5\cdot 10^{-2}$ & $3.1\cdot 10^{-1}$ & $3.1\cdot 10^{-1}$ & Prescribed \\
23 & $1.2 -2.2\cdot 10^{-15}i$ & $1.5\cdot 10^{-2}$ & $3.4\cdot 10^{-1}$ & $3.4\cdot 10^{-1}$ & Prescribed \\
24 & $-0.47-0.03i$ & $4.0\cdot 10^{-4}$ & $2.0\cdot 10^{-16}$ & $3.1\cdot 10^{-3}$ & Random \\
25 & $-0.47 +0.03i$ & $4.6\cdot 10^{-4}$ & $3.0\cdot 10^{-3}$ & $1.3\cdot 10^{-15}$ & Random \\
\hline
\end{tabular}}
\end{table}

Similarly, The $\mu$-components are the eigenvalues of the pencil $(\Delta_2, \, \Delta_0) = (A_1 \otimes B_2 - A_2 \otimes B_1, \, B_1 \otimes C_2 - C_1 \otimes B_2)$.
The 9 roots $(\lambda, \mu)$ are
\begin{align*}
& \\[-7mm]
& (-2.4183, \ 1.8542), \\
& (-0.5609 \pm 2.0355i, \ \ph-1.6092 \pm 0.3896i), \\
& (-1.1331 \pm 0.3012i, \ -0.3845 \mp 0.9454i), \\
& (\ph-0.0807 \pm 1.1123i, \ -1.0874 \pm 0.1905i), \\
& (\ph-0.0724 \pm 1.2249i, \ -0.3144 \mp 1.1038i).
\end{align*}
\end{experiment}%
In conclusion, for low-degree polynomials, this approach gives symmetric determinantal representations of almost the same size, and, as a result, symmetric $\Delta$-matrices for which we now have perturbations preserving a Hermitian structure.

\section{Conclusions} \label{sec:concl}
Extending results from Parts I and II \cite{HMP19,HMP_SingGep2} for unstructured singular pencils, we have proposed three structure-preserving methods to compute eigenvalues of singular real symmetric or complex Hermitian pencils.
This means that the perturbed pencils $\widetilde A-\lambda\widetilde B$ remain Hermitian.
If the original $B$ is positive semidefinite, then $\widetilde B$ can be chosen positive definite (see Remark~\ref{rm:definite}).
This can be exploited by eigensolvers, so that, for instance, the computed eigenvalues are real without any side-effects of finite precision arithmetic.

Just as in the general nonsymmetric situation \cite{HMP19, HMP_SingGep2}, there are three options: a rank-completing perturbation of size $k$, projection onto the normal rank of size $r=n-k$, or augmentation by dimension $k$. In this paper we have concentrated on the perturbation approach and its analysis in Section~\ref{sec:rankcomplete}, and on the projection approach in Section~\ref{sec:proj}.
We have shown that a Hermitian perturbation then leads to generically simple random eigenvalues.
Augmentation with dimension $k$, which extends \cite{HMP_SingGep2} for the nonsymmetric case, may be useful for large sparse problems, since this technique may (partly) preserve sparsity.
Generically this will lead to a regular symmetric (or Hermitian) pencil of size $n+k$; see Remark~\ref{rem:augm}.
For situations in which we might possibly overestimate or underestimate the normal rank, we refer to the discussion in Part II \cite[Sec.~8]{HMP_SingGep2}.
For pseudocodes of the methods, we refer to Parts I and II, where the main difference is that we now use a complex $U=V$.

We highlight again that the structure-preserving methods presented in this paper may also be
useful in the case that they are post-processed by standard eigensolvers that
ignore the special symmetry structure of the problem. Indeed, it has been discussed in
Remark~\ref{rem:sign} that the sign characteristic of the pencil can be obtained from
the perturbed Hermitian pencils once its eigenvectors have been computed.

The case for real symmetric pencils with real symmetric perturbations (see Conjecture~\ref{conj:double}) turns out
to be challenging and is left for future research.

%\bibliographystyle{abbrv}
%\bibliography{refs, refs_BP, refs_MEP}

\end{document}